\documentclass[11pt,reqno, a4paper]{amsart}
\usepackage{marginnote}
\usepackage{textcomp,  microtype, lmodern, mathrsfs,fbb}
\usepackage{todonotes}
\usepackage[english]{babel}
\usepackage{amsthm,amsmath}
\usepackage{amssymb,bbm,esint}
\numberwithin{equation}{section}
\usepackage[dvipsnames]{xcolor}
\usepackage{comment}

\usepackage{dutchcal}
\usepackage{bbm}
\usepackage{tikz}
\usetikzlibrary{decorations.pathreplacing}
\usepackage[T1]{fontenc}
\usepackage[utf8]{inputenc}
\usepackage{csquotes}
\usepackage{mathtools,apptools}
\usepackage{pdfpages}
\usepackage[top=.8in, bottom=.8in, left=1.1in, right=1.1in]{geometry}
\usepackage{fancyhdr}
\usepackage{nicefrac}
\usepackage{tikz-cd}
\usepackage{esint,dsfont}

\usepackage[backref=page]{hyperref}
\hypersetup{breaklinks=true,hypertexnames=false,colorlinks=true,urlcolor=ForestGreen,citecolor=ForestGreen,linkcolor=blue!70!green,bookmarksnumbered,bookmarksopen}
\usepackage{orcidlink}

\usepackage[capitalise]{cleveref}
\newtheorem{theorem}{Theorem}[section]
\newtheorem{definition}[theorem]{Definition}
\newtheorem{lemma}[theorem]{Lemma}
\newtheorem{remark}[theorem]{Remark}
\newcommand{\hn}{{\mathbb{H}^{N}}}

 \newcommand{\<}{\left\langle}
\renewcommand{\>}{\right\rangle}
\newcommand{\eps}{\varepsilon}
\newcommand{\dhn}[2]{|#2^{-1} \circ #1|_\hn}

\usepackage{longtable}
\title[Equivalence of weak and viscosity solutions]{Equivalence of Weak and Viscosity Solutions for Nonlocal p-Laplace Type Equations on the Heisenberg Group}
\author{Debangana Mukherjee and Vivek Tewary}
\address{School of Interwoven Arts and Sciences, Krea University, Sri City, India}
\email{debangana.mukherjee@krea.edu.in, vivek.tewary@krea.edu.in}
\begin{document}
	
	\begin{abstract}
    We establish the equivalence between weak and viscosity solutions for a broad class of nonlocal $p$-Laplace type equations on the Heisenberg group whose kernels satisfy standard symmetry, ellipticity, and left-translation invariance assumptions. As a particular case, our results apply to the fractional Heisenberg $p$-Laplacian. The proof combines intrinsic approximation by modified infimal convolutions, Heisenberg mollification, and a nonlocal integration-by-parts argument adapted to the sub-Riemannian geometry. The main analytical difficulty stems from the noncommutative structure of the Heisenberg group, which prevents a direct extension of the Euclidean theory and requires new localization and approximation techniques. 
	\end{abstract}

	\maketitle
	
	\section{Introduction}
Nonlocal nonlinear equations have become a central topic in the analysis of partial differential equations owing to their broad range of applications and their rich mathematical structure. Among these, fractional $p$-Laplace type equations constitute a fundamental class of nonlinear integro-differential equations, combining the features of degenerate elliptic operators with long-range interactions. Their study lies at the intersection of nonlinear potential theory, the calculus of variations, and viscosity solution theory, and has led to advances in regularity theory, comparison principles, existence and uniqueness of solutions, and the qualitative analysis of nonlinear nonlocal phenomena. 

Two complementary notions of solutions play a central role in the analysis of nonlinear nonlocal equations. Weak solutions arise naturally from the variational structure of the problem and are well suited to energy estimates and compactness methods. Viscosity solutions, on the other hand, provide a pointwise framework compatible with comparison principles and stability arguments. Establishing the equivalence between these notions allows one to transfer techniques between the variational and viscosity settings.

In the Euclidean setting, Korvenpää, Kuusi and Lindgren in \cite{Korven2019} established the equivalence of weak, viscosity and potential-theoretic solutions for a broad class of fractional $p$-Laplace type operators. More recently, Barrios and Medina in \cite{Medina2021} extended this theory to non-homogeneous equations, developing new approximation arguments based on infimal convolutions and weak compactness. The analysis of nonlocal equations has further expanded to encompass rough forcing terms and measure data. In this direction, Kuusi, Mingione and Sire \cite{Kuusi2015} established a nonlinear potential-theoretic approach for nonlocal equations with measure data. Further, Di Castro, Kuusi and Palatucci \cite{Kuusi2016} established nonlocal Harnack inequalities for fractional \(p\)-Laplace type equations. Brasco et al \cite{Brasco2023} proved higher Hölder regularity for solutions to the fractional \(p\)-Laplacian. Brasco, Lindgren and Parini \cite{Brasco2016} introduced and studied the fractional Cheeger problem. For works on the regularity theory of quasilinear nonlocal equations on the Heisenberg group, see \cite{ManfrediniPalatucciPiccininiPolidoro2023,Piccinini2022,PalatucciPiccinini2023}. 

The primary objective of this paper is to establish the equivalence between weak and viscosity solutions for fractional $p$-Laplace type equations on the Heisenberg group. Motivated by the Euclidean theory developed by Korvenpää, Kuusi and Lindgren \cite{Korven2019} and subsequently extended by Barrios and Medina \cite{Medina2021}, we investigate nonlocal operators whose prototype is
    \begin{align}
        \label{maineq}
                2\,\text{P.V.}\int_{\mathbb{H}^N}\frac{|u(\xi)-u(\eta)|^{p-2}(u(\xi)-u(\eta))}{|\eta^{-1}\circ\xi|^{Q+sp}}\,d\eta = f(\xi),
    \end{align}
in the sub-Riemannian setting of the Heisenberg group. The analysis developed in this paper applies to a general class of nonlocal \(p\)-Laplace type operators. More precisely, in Section \ref{sec:kernel}, we introduce a kernel \(K(\xi,\eta)\) satisfying suitable ellipticity, symmetry, left-translation invariance and continuity assumptions, and formulate the operator in this general setting. The fractional \(p\)-Laplacian corresponds to the particular choice
\[
K(\xi,\eta)=|\eta^{-1}\circ\xi|_{\mathbb H^N}^{-(Q+sp)}.
\]
One of the principal difficulties in extending the Euclidean theory to the Heisenberg group stems from the intrinsic sub-Riemannian geometry and the noncommutative structure of the underlying Lie group. Unlike the Euclidean setting, the natural notion of differentiability is given by the horizontal vector fields, while classical translations, Taylor expansions, and convolution arguments must be replaced by their counterparts. In particular, the analysis of the nonlocal operator requires the use of group translations, homogeneous dilations, the Heisenberg gauge, and modified infimal convolutions. These analytical and geometric obstacles prevent a straightforward adaptation of the Euclidean proofs and require several new technical ingredients that are developed in this paper. Although the theory of local subelliptic equations on the Heisenberg group is now well established, thanks to the foundational works in \cite{Folland1975,FollandStein1982,Bonfig2007,Capogna2007,Bieske2002,Bieske2006,OchoaRuiz2019}, the corresponding theory for nonlinear nonlocal equations remains comparatively less developed.
\subsection*{Main contributions and novelties}
The main contributions of the present paper are summarized as follows.
\begin{itemize}
\item We establish the equivalence between weak and viscosity solutions for a general class of nonlocal \(p\)-Laplace type equations on the Heisenberg group. To the best of our knowledge, this is the first such equivalence result in the sub-Riemannian nonlocal setting.
\item We introduce a modified infimal convolution adapted to the noncommutative structure of the Heisenberg group. Unlike the classical Euclidean infimal convolution, the regularization is constructed using the intrinsic group operation, preserving the left-translation invariance of the equation.
\item We establish a principal value formula for affine functions in the Heisenberg group, which plays a fundamental role in the viscosity framework and allows the nonlocal operator to be interpreted pointwise.
\item We develop an approximation procedure that allows us to pass from viscosity supersolutions to weak supersolutions. The proof combines the regularization by infimal convolution, uniform local energy estimates, and compactness arguments adapted to the Heisenberg group.
\end{itemize}

For the convenience of the reader, we briefly recall the basic geometric and analytic framework of the Heisenberg group that underlies the analysis developed in this paper. Unless otherwise stated, the material in this subsection follows the standard references \cite{Bonfig2007,FollandStein1982}.
    
    \subsection{Preliminaries on the Heisenberg Group} 
    The Heisenberg group is the set $\mathbb{R}^{2N+1}$ equipped with the following group multiplication for $$\xi=(z,t)=(x,y,t)=(x_1, x_2, \ldots, x_N, y_1, y_2,\ldots, y_N, t)$$ and $$\xi':=(z',t')=(x',y',t')=({x'}_1, {x'}_2, \ldots, {x'}_N,{y'}_1, {y'}_2,\ldots, {y'}_N, t'),$$ 
	\begin{align*}
		\xi\circ\xi' = (x+x', y+y', t+t'+2\langle x',y\rangle-2\langle x,y'\rangle).
	\end{align*}
	One defines a one-parameter group of automorphisms on $\hn$ as $\Phi_\lambda(x,y,t)=(\lambda x,\lambda y,\lambda^2 t)$ so that it is said to have a \emph{homogeneous dimension} of $Q=2N+2$. 
    A \emph{homogeneous norm} $d_0:\hn\to[0,\infty]$ is a function satisfying
	\begin{enumerate}
		\item $d_0(\Phi_\lambda(\xi))=\lambda d_0(\xi)$ for any $\xi\in\hn$ and $\lambda>0$.
		\item $d_0(\xi)=0$ if and only if $\xi=0$.
	\end{enumerate}
	One defines the standard \emph{homogeneous norm} on $\hn$ by 
	\[
	|\xi|_\hn = (|z|^4+|t|^2)^{\tfrac{1}{4}}.
	\] It is well known that all homogeneous norms on $\hn$ are equivalent (see \cite[Corollary 5.1.5]{Bonfig2007}). For a center $\xi_0\in\hn$ and radius $R>0$, we define the ball $B(\xi_0,R):=\{\xi\in\hn: |\xi_0^{-1}\circ\xi|<R\}$. 
	Moreover, Heisenberg group with any homogeneous norm $d_0$ satisfies a pseudo-triangle inequality in the form of the lemma below:
	\begin{lemma} \cite[Proposition 5.1.7]{Bonfig2007}
		Let $d_0$ be a homogeneous norm on $\hn$. Then there exists a constant $c>0$ such that for all $\xi,\eta\in\hn$, it holds that
		\begin{enumerate}
			\item $d_0(\xi\circ\eta)\leq c(d_0(\xi)+d_0(\eta))$,
			\item $d_0(\xi\circ\eta)\geq\frac1c d_0(\xi)-d_0(\eta^{-1})$,
			\item $d_0(\xi\circ\eta)\geq\frac1c d_0(\xi)-c d_0(\eta)$.
		\end{enumerate}
	\end{lemma}
	\begin{remark}
		It is known that for the standard homogeneous norm on $\hn$, the constant $c$ above is $1$ so we may assume that $\hn$ is a metric space with metric given by $|\cdot^{-1}\circ\cdot|$ (See \cite[Example 5.1]{BFS18}).
	\end{remark}

    An equivalent homogeneous norm may be defined as in \cite{BCM96}. For any $\xi\in\hn\setminus\{0\}$, define $\|\xi\|_\hn=\rho$ where $\rho$ is the unique solution of the equation
    \begin{align*}
        \frac{x_1^2}{\rho^2}+\cdots+\frac{x_N^2}{\rho^2}+\frac{y_1^2}{\rho^2}+\cdots+\frac{y_N^2}{\rho^2}+\frac{t^2}{\rho^4}=1.
    \end{align*}
    Moreover, define ``polar-type" coordinates
    \begin{align*}
        \begin{cases}
            x_1&=\rho\cos\psi_1\ldots\cos\psi_{2N-1}\cos\psi_{2N},\\
            x_2&=\rho\cos\psi_1\ldots\cos\psi_{2N-1}\sin\psi_{2N},\\
            \vdots\\
            x_N&=\rho\cos\psi_1\cos\psi_2\cdots\cos\psi_{N-1}\cos\psi_{N}\sin\psi_{N+1},\\
            y_1&=\rho\cos\psi_1\cos\psi_2\cdots\cos\psi_{N-1}\sin\psi_{N},\\
            \vdots\\
            y_N&=\rho\cos\psi_1\sin\psi_2,\\
            t&=\rho^2\sin\psi_1,
        \end{cases}
    \end{align*} with \begin{align}\label{eq:polar}d\xi = \rho^{Q-1} J(\psi_1,\ldots,\psi_{2N})d\rho\,d\psi_1\cdots d\psi_{2N}=\rho^{Q+1}d\rho\, d\sigma.\end{align}
    
	We also note that the Haar measure on $\hn$ is the same as the Lebesgue measure on $\mathbb{R}^{2N+1}$ and it satisfies the measure doubling property, viz.,
	\begin{align*}
		\left|B(\xi_0,2R)\right|\leq C \left|B(\xi_0,R)\right|
	\end{align*} for any $\xi_0\in\hn$ and $R>0$. These properties make $(\hn,|\cdot^{-1}\circ\cdot|,\mathcal{L}^{2N+1})$ into a $Q$-regular measure metric space with a doubling measure. These results may be found in \cite[Proposition 14.2.9]{Shanmugalingam2015} and \cite[Section 11.3]{Koskela2000}.

    Now, we turn to the existence of Taylor series for smooth functions on the Heisenberg group. The Jacobian basis of the Heisenberg Lie algebra $\mathfrak{h}^N$ of $\hn$ is given by $X_j=\partial_{x_j} + 2y_j\partial_t$ for $1\leq j\leq N$, $X_{N+j} = \partial_{y_j} - 2x_j\partial_t$ for $1\leq j\leq N$, and $X_{2N+1}=T=\partial_t$.

    \begin{definition}(Maclaurin Polynomial~\cite[Def. 20.2.3]{Bonfig2007})
        Let $u\in C^\infty(\hn,\mathbb{R})$. Then for any $n\in \mathbb{N}\cup\{0\}$, there exists a unique polynomial $p$, $\Phi_\lambda$-homogeneous of degree at most $m$ such that 
        \begin{align*}
            (X_1\dots X_{2n}T)^\beta p(0) = (X_1\dots X_{2n}T)^\beta u(0)
        \end{align*} for all multiindices $\beta = (\beta_1,\ldots,\beta_{2n},\beta_{2n+1})$ with $|\beta| = \beta_1+\cdots+\beta_{2n}+\beta_{2n+1} \leq m$.
    \end{definition}

    We say that $P := P_m(u,0)(\xi)$ is the Maclaurin polynomial of $\Phi_\lambda$-degree $m$ associated to $u$. In the case of the Heisenberg group, one can explicitly write the Maclaurin polynomial of $\Phi_\lambda$-degree \(2\). We have
    \begin{align}
        P_2(u,0)(z, t)=u(0)+\langle \nabla_{\mathbb{H}^N} u(0), z\rangle+\partial_t u(0)\, t+\frac12 \langle z, D^{2,*}_{\mathbb{H}^N} u(0)\, z\rangle .
    \end{align}
    Here the subgradient $\nabla_{\mathbb{H}^N}u$ is given by
    \begin{align*}
    \nabla_{\mathbb{H}^N} u(\xi):=\left( X_1 u(\xi), X_2 u(\xi),\dots, X_{2n} u(\xi) \right),
    \end{align*} and \(D^{2,*}_{\mathbb{H}^N} u\) is the symmetrized horizontal Hessian matrix, i.e.,
    \begin{align*}
        D^{2,*}_{\mathbb{H}^N} u(\xi):=\left(\frac12\big( X_i X_j u(\xi) + X_j X_i u(\xi) \big)\right)_{i,j=1,\dots,2N}.
    \end{align*}

    We will use the notation $\nabla^{2,*}_\hn u(\xi)=\left(\frac12\big( X_i X_j u(\xi) + X_j X_i u(\xi) \big)\right)_{i,j=1,\dots,2N+1}$ to denote the full Hessian.

    \begin{definition}
        Let $u \in C^{m}( \mathbb{H}^N , \mathbb{R}),\; \xi_0 \in \mathbb{H}^N$ and $m \in \mathbb{N}\cup\{0\}$. We consider the Maclaurin polynomial \(P_m(u(\xi_0 \circ \cdot),0)\) of the map 
        \begin{align*}
            &\xi \longmapsto u(\xi_0 \circ \xi),\\
            &\mathbb{H}^N \to \mathbb{R}.
        \end{align*}
    The polynomial
    \begin{align*}
        P_m(u,\xi_0)(\xi) := P_m\big(u(\xi_0 \circ \cdot),0\big)\big(\xi_0^{-1}\circ \xi\big)
    \end{align*} is the Taylor polynomial of $\Phi_\lambda$-degree \(m\) centered at \(\xi_0\) associated to \(u\).
    \end{definition}
The following intrinsic Taylor theorem will be the main analytical tool in the proof of the equivalence between weak and viscosity solutions, allowing us to control the singular behaviour of the nonlocal operator near the diagonal.
 \begin{theorem}\label{thm:taylor} (\cite[Corollary 20.3.5]{Bonfig2007})
        Let \(m \in \mathbb{N}\cup\{0\}\), \(\xi_0 \in \mathbb{H}^N\).  Let \(u \in C^{m+1}(\mathbb{H}^N,\mathbb{R})\). Then it holds that 
        \begin{align*}
          u(\xi)=P_m(u,\xi_0)(\xi)+O\!\left(|\xi_0^{-1}\circ\xi|_{\mathbb H^N}^{\,m+1}\right).
        \end{align*}
    \end{theorem}
    \subsection{Nonlocal operators and structural assumptions} \label{sec:kernel}
    Let $K:\hn\times\hn\setminus\{0,0\}\to (0,\infty)$ be a function satisfying
    \begin{align}\label{eq:Kernel}
        \frac{\lambda}{\dhn{\xi}{\eta}^{Q+sp}}\leq K(\xi,\eta)\leq \frac{\Lambda}{\dhn{\xi}{\eta}^{Q+sp}} \mbox{ for } (\xi,\eta)\neq (0,0),
    \end{align} where $0<\lambda\leq \Lambda$, $p>1$ and $s\in (0,1)$.
    We further assume
    \begin{enumerate}
        \item Symmetry: $K(\xi,\eta)=K(\eta,\xi)$ for all $\xi,\eta\in\hn$.
        \item Left-translation invariance: $$K(\xi_0\circ\xi,\xi_0\circ\eta)=K(\xi,\eta) \mbox{ for all } \xi_0,\xi,\eta\in\hn.$$
        \item Continuity: The map $\xi\mapsto K(\xi,\eta)$ is continuous in $\hn\setminus\{0\}$.
    \end{enumerate}

    These properties hold for the kernel $K(\xi,\eta)=\dhn{\xi}{\eta}^{-Q-sp}$.
    
    The operator of our interest is
    \begin{align}\label{eq:operator}
        \mathcal{L}u(\xi):=2 \text{P.V.}\int_{\mathbb{H}^N}|u(\xi)-u(\eta)|^{p-2}(u(\xi)-u(\eta))K(\xi,\eta)\,d\eta
    \end{align}
   \begin{remark} Throughout this paper, we work with the nonlocal operator \eqref{eq:operator}
where the kernel \(K\) satisfies the assumptions in (\ref{eq:Kernel}). The fractional \(p\)-Laplacian on the Heisenberg group is recovered as the particular choice
\[
K(\xi,\eta)=|\eta^{-1}\circ\xi|_{\mathbb H^N}^{-(Q+sp)}.
\]
Therefore, all the results established in this paper for the general class of nonlocal \(p\)-Laplace type operators apply, in particular, to the fractional \(p\)-Laplacian on the Heisenberg group.
\end{remark}

    We define the fractional Sobolev space $W^{s,p}(\Omega)$ as
	\begin{align*}
		W^{s,p}(\Omega):=\left\{u\in L^p(\Omega):\int_\Omega\int_\Omega\frac{|u(\xi)-u(\eta)|^p}{|\eta^{-1}\circ\xi|_\hn^{Q+sp}}\,d\xi\,d\eta<\infty\right\}.
	\end{align*}
    We denote by $W_0^{s,p}(\Omega)$ the closure of $C^\infty_0(\Omega)$ in $W^{s,p}(\hn)$. Now we define a tail space $L^m_\gamma(\hn)$ on Heisenberg Group as
	\begin{align*}
		L^m_\gamma(\hn):=\Bigg\{v\in L^m_{\textup{loc}}(\hn):\int_\hn\frac{|v|^m}{1+|\eta|_\hn^{Q+\gamma}}d\eta<+\infty\Bigg\},m>0,\gamma>0.
	\end{align*}
    
    The tail space is defined in consideration of the tail of the function $u$ which is defined to be
    \begin{align}
        \textup{Tail}(f;\xi,r):=\left( r^{sp} \int_{\hn\setminus B_r(\xi)} \frac{|f(\eta)|^{p-1}}{|\xi^{-1}\circ\eta|_\hn^{Q+ps}} \,d\eta\right)^{\frac{1}{p-1}}, r>0.
    \end{align}

    \begin{definition}
        Let $f\in L^p_{\text{loc}}(\Omega)$. A function $u\in W^{s,p}_{\text{loc}}(\Omega)\cap L^{p-1}_{sp}(\hn)$ is said to be a weak supersolution (subsolution) of \eqref{maineq} if
        \begin{equation}
        \begin{aligned}
            \iint_{\hn\times\hn} |u(\xi)-u(\eta)|^{p-2}(u(\xi)-u(\eta))(\phi(\xi)-\phi(\eta))K(\xi,\eta)\,d\eta\,d\xi\\
            \geq (\leq) \int_\hn f(\xi)\phi(\xi)\,d\xi
        \end{aligned}
        \end{equation} for all non-negative $\phi\in C_0^\infty(\hn)$.
    \end{definition}

    In order to define the notion of viscosity solution for exponents in the range $p\leq\frac{2}{2-s}$, we need a more restricted class of comparison functions. Given an open set $D$, we denote by $C^2_{\beta}(D)$, a subset of $C^2(D)$, defined as
    \begin{align}
        C^2_{\beta}(D) = \left\{\phi\in C^2(D) : \sup_{\xi\in D}\left[\frac{\min\{d_\phi(\xi),1\}^{\beta-1}}{|\nabla_\hn\phi(\xi)|}+\frac{|D^2_{\hn}\phi(\xi)|}{d_\phi(\xi)^{\beta-2}}\right] <\infty\right\},
    \end{align} where $d_\phi(\xi) = \text{dist}(\xi, N_\phi)$ and $N_\phi=\{\xi\in D:\nabla_{\hn}\phi(\xi)=0\}$. We denote the supremum in this definition by $\|\cdot\|_{C^2_\beta(D)}$. The following definition is adapted from \cite{Korven2019}.

    \begin{definition}\label[definition]{def-visc}
        A function $u:\hn\to[-\infty,\infty]$ is a viscosity supersolution (subsolution) to $\mathcal{L}u=f$ in $\Omega$ if it satisfies the following:
        \begin{enumerate}
        \item $u<+\infty (u > -\infty)$ a.e. in $\hn$, $u>-\infty (u<+\infty)$ a.e. in $\Omega$.
            \item The function $u$ is lower(upper) semicontinuous in $\Omega$.
            \item If $\phi\in C^2(B_r(\xi_0))\cap L^{p-1}_{sp}(\hn)$ for some $B_r(\xi_0)\subset\Omega$ such that $\phi(\xi_0)=u(\xi_0)$, $\phi\leq u$ $(\phi\geq u)$ in $\hn$ and one of the following conditions hold:
            \begin{enumerate}
                \item $p>\frac{2}{2-s}$ or $\nabla_{\hn}\phi(\xi_0)\neq 0$,
                \item $1<p\leq \frac{2}{2-s}$, $\nabla_{\hn}\phi(\xi_0)=0$ is such that $\xi_0$ is an isolated critical point of $\phi$, and $\phi\in C^2_{\beta}(B_r(\xi_0))$ for some $\beta>\frac{sp}{p-1}$,
            \end{enumerate} then we have $\mathcal{L}\phi_r(\xi_0)\geq (\leq) f(\xi_0)$ where
            \begin{align*}                    
            \phi_r(\xi)=
               \begin{cases}
               \phi(\xi) &\mbox{ for }\xi\in B_r(\xi_0),\\
               u(\xi)&\mbox{ otherwise.}
                \end{cases}
            \end{align*}
            \item $u_-$ (respectively $u_+)\in L^{p-1}_{sp}(\hn)$.
        \end{enumerate}
    \end{definition}
    
    \subsection{Main Theorems}
    We will prove the following two theorems.

    \begin{theorem}
        \label{maintheorem1} Let $1<p<\infty$. Assume that $f\in C(\Omega)$. Let $u\in W^{s,p}_{\text{loc}}(\Omega)\cap L^{p-1}_{sp}(\hn)$ be  a weak subsolution to $\mathcal{L}u=f$ in $\Omega$. Furthermore, assume that $u$ is upper-semicontinuous in $\Omega$. Then $u$ is a viscosity subsolution to $\mathcal{L}u=f$ in $\Omega$. The analogous conclusions hold for weak supersolution and solution.
    \end{theorem}
        \begin{theorem} \label{maintheorem2}
Let $s\in(0,1)$ and assume that $p>\frac{2}{2-s}.$ Let $f\in C(\Omega)$, and let
$u\in L^\infty(\mathbb H^N)$ be a viscosity subsolution of $\mathcal Lu=f \quad\text{in }\Omega.$ Then $u$ is a weak subsolution of $\mathcal Lu=f\text{ in }\Omega.$
The analogous conclusions hold for viscosity supersolutions and viscosity solutions.
\end{theorem}

    \subsection{Plan of the Paper} In \cref{sec:aux}, we collect some auxilliary results related to algebraic inequalities and properties of the principal value. In \cref{sec:comp}, we prove the comparison principle for solutions of the $p$-Laplace equation on the Heisenberg group. In \cref{sec:thm1}, we prove that weak solutions are viscosity solutions and in \cref{sec:thm2}, we prove that viscosity solutions are weak solutions.

    \section{Auxilliary Results} \label{sec:aux}
    \subsection{Algebraic inequalities} In this section, we will prove certain algebraic identities related to the $p$-structure. The proofs are motivated from \cite{Korven2019} even as they require working with the homogeneous norm of the Heisenberg group.
We begin by establishing that the principal value of the nonlocal operator vanishes on affine functions, a property that will play a fundamental role in the localization arguments developed later.
    \begin{lemma}\label[lemma]{lem:affine}(Vanishing principal value for affine functions)
    Let $l : \mathbb{H}^N \to \mathbb{R}$ be a Euclidean affine function of the form $l(\xi) = a+b\cdot z+ct$ where $b\in\mathbb{R}^{2N}$, $a,c\in\mathbb{R}$ and $b\cdot z$ represents the Euclidean dot product of $b$ and $z$ and let $0<r<\infty$.
    Then
    \begin{align*}
    \int_{B_r(\xi)\setminus B_{\varepsilon}(\xi)} |l(\xi)-l(\eta)|^{p-2} \, (l(\xi)-l(\eta))\, K(\xi,\eta)\, d\eta = 0,
    \end{align*} for all $0< \varepsilon < r$.
    \end{lemma}

    \begin{proof}
        By left-translation, note that
        \begin{equation*}
        \begin{aligned}
    &\int_{B(\xi,r)\setminus B(\xi,\varepsilon)} |l(\xi)-l(\eta)|^{p-2} \, (l(\xi)-l(\eta))\, K(\xi,\eta)\, d\eta = \\ & \quad \int_{B(0,r)\setminus B(0,\varepsilon)} |l(\xi)-l(\xi\circ\zeta)|^{p-2} \, (l(\xi)-l(\xi\circ\zeta))\, K(\xi,\xi\circ\zeta)\, d\zeta. 
    \end{aligned}
    \end{equation*}
    Then, by definition of $l$, for $\zeta = (z',t')=(x',y',t')$ we have
    
    \resizebox{\linewidth}{!}{
    \begin{minipage}{\linewidth}
    \begin{align*}
        g(\xi)&:=\int_{B(0,r)\setminus B(0,\varepsilon)} |l(\xi)-l(\xi\circ\zeta)|^{p-2} \, (l(\xi)-l(\xi\circ\zeta))\, K(\xi,\xi\circ\zeta)\, d\zeta\\
        & = -\int_{B(0,r)\setminus B(0,\varepsilon)} |b\cdot z'+c(t'+2(x'\cdot y - y'\cdot x))|^{p-2} \, (b\cdot z'+c(t'+2(x'\cdot y - y'\cdot x)))\, K(\xi,\xi\circ\zeta)\, d\zeta\\
        & = -\int_{B(0,r)\setminus B(0,\varepsilon)} |b\cdot z'+c(t'+2(x'\cdot y - y'\cdot x))|^{p-2} \, (b\cdot z'+c(t'+2(x'\cdot y - y'\cdot x)))\, K(0,\zeta)\, d\zeta\\
        & = -\int_{B(0,r)\setminus B(0,\varepsilon)} |-b\cdot z'-c(t'+2(x'\cdot y - y'\cdot x))|^{p-2} \, (-b\cdot z'-c(t'+2(x'\cdot y - y'\cdot x)))\, K(0,-\zeta)\, d\zeta\\
        & = -\int_{B(0,r)\setminus B(0,\varepsilon)} |-b\cdot z'-c(t'+2(x'\cdot y - y'\cdot x))|^{p-2} \, (-b\cdot z'-c(t'+2(x'\cdot y - y'\cdot x)))\, K(0,\zeta^{-1})\, d\zeta\\
        & = -\int_{B(0,r)\setminus B(0,\varepsilon)} |-b\cdot z'-c(t'+2(x'\cdot y - y'\cdot x))|^{p-2} \, (-b\cdot z'-c(t'+2(x'\cdot y - y'\cdot x)))\, K(\zeta,0)\, d\zeta\\
        & = \phantom{-}\int_{B(0,r)\setminus B(0,\varepsilon)} |b\cdot z'+c(t'+2(x'\cdot y - y'\cdot x))|^{p-2} \, (b\cdot z'+c(t'+2(x'\cdot y - y'\cdot x)))\, K(0,\zeta)\, d\zeta\\
        &=-g(\xi),
    \end{align*}
    \end{minipage}}
    where in the second step, we have used the definition of $l$, in the third step, we have used left-translation invariance of $K$; in the fourth step, we perform the transformation $\zeta\mapsto-\zeta$; in the fifth step, we use the fact that $\zeta^{-1} = -\zeta$; in the sixth step, we use left-translation invariance of $K$; and in the seventh step, we use the symmetry of $K$.
    \end{proof}

 The proofs of the following algebraic identities can be found in \cite[Lemma~3.2-3.4]{Korven2019}. These lemma collects several fundamental algebraic inequalities satisfied by the nonlinear map \(t \mapsto |t|^{p-2}t\), which will be used repeatedly throughout the paper in the analysis of the nonlocal operator and the derivation of the main estimates.
\begin{lemma}\label[lemma]{lem:algebraic}(Algebraic inequalities for the p-structure)
(\cite[Lemma~3.2-3.4]{Korven2019},\cite[Lemma~2.4]{Medina2021}) Let $p>1$, $a,b\in \mathbb{R}$. Then     \begin{enumerate}
 \item it holds that
\begin{align*}
     c_p (|a|+|b|)^{p-2} \le \int_0^1 |\, a+t(b-a)\,|^{p-2}\, dt \le C_p (|a|+|b|)^{p-2},
 \end{align*}
  where
\begin{align*}
     c_p =   \begin{cases}
                    1, & 1<p\le 2,\\
                    \dfrac{4^{2-p}}{p-1}, & 2\le p,
                \end{cases}
    \qquad
        C_p =   \begin{cases}
                    \dfrac{4^{2-p}}{p-1}, & 1<p<2,\\
                    1, & 2\le p.
                \end{cases}
    \end{align*}
    
    \item  it holds that
    \begin{align*}
        \bigl|\, |a|^{p-2} a - |b|^{p-2} b \,\bigr|\;\le\;C \bigl( |b| + |a-b| \bigr)^{p-2} |a-b|
    \end{align*}
    for some $C = C(p)>0$.
    \item it holds that
    \begin{align*}
        |a|^{p-2}a-|b|^{p-2}b = (p-1)(a-b)\int_0^1 |\, ta+(1-t)b)\,|^{p-2}\, dt
    \end{align*}
    \end{enumerate}
    \end{lemma}
    We next derive an auxiliary integral estimate on the homogeneous sphere, which will be useful in proving the well-definedness of the nonlocal operator.
    \begin{lemma}\label[lemma]{lem:sphere}(Angular integral estimate on the homogeneous sphere)
    Let $e\in \mathbb{R}^{2N}$ be such that $|e|_{\mathbb{R}^{2N}}=1$, $p>1$ and $a\geq 0$.
    Then
    \begin{align*}
        \int_{S(0,1)} \bigl( |(e\cdot z_\sigma)| + a\bigr)^{p-2} d\sigma\;\le\; C (1+a)^{p-2}
    \end{align*} for some $C = C(Q,p)>0$, where $d\sigma$ is as in \cref{eq:polar}, $S(0,1)$ denotes the unit sphere in $\mathbb{R}^{2N+1}$, and $z_\sigma$ denotes the horizontal component of $\sigma\in S(0,1)$.
    \end{lemma}
 
 \begin{proof}
Let $\sigma\in S(0,1)$ and write
\begin{equation*}
\sigma=(z_\sigma,t_\sigma)\in \mathbb R^{2N}\times \mathbb R.
\end{equation*}
Using the polar coordinates \eqref{eq:polar}, we write
\[
z_\sigma=\cos\psi_1\,\omega,
\]
where \(\omega\in S^{2N-1}\subset\mathbb R^{2N}\). Since \(\rho=1\) on \(S(0,1)\), we have
\[
|z_\sigma|=|\cos\psi_1|,
\qquad
t_\sigma=\sin\psi_1.
\]
By the rotational invariance of the horizontal variable $z_\sigma$, we may
assume without loss of generality that
$e=(1,0,\ldots,0)\in \mathbb R^{2N}.$
Therefore, 
$|e\cdot z_\sigma|=|x_1|,$ and
\begin{gather*}
x_1=\cos\psi_1\,\omega_1,
\end{gather*}
where $\omega_1$ is the first coordinate of $\omega$.

From the polar-coordinate formula \eqref{eq:polar}, the surface measure on
$S(0,1)$ has the form
$$
d\sigma
=
J(\psi_1,\ldots,\psi_{2N})\,d\psi_1\cdots d\psi_{2N}.
$$
Equivalently, after separating the first angular variable from the
Euclidean angular variables on $\mathbb S^{2N-1}$, we may write
\begin{align*}
d\sigma
\le
C(Q)\,
|\cos\psi_1|^{2N-1}\,d\psi_1\,d\omega.
\end{align*}
Therefore,
\[
\begin{aligned}
\int_{S(0,1)}
\bigl(|e\cdot z_\sigma|+a\bigr)^{p-2}\,d\sigma
&\le
C(Q)
\int
\int_{\mathbb S^{2N-1}}
\bigl(|\cos\psi_1|\,|\omega_1|+a\bigr)^{p-2}
|\cos\psi_1|^{2N-1}
\,d\omega\,d\psi_1.
\end{aligned}
\]

We now distinguish two cases.

First assume $p\ge 2$. Since
$
|e\cdot z_\sigma|\le |z_\sigma|\le 1,
$
we have
\begin{equation*}
|e\cdot z_\sigma|+a\le 1+a.
\end{equation*}
Hence
$$
\bigl(|e\cdot z_\sigma|+a\bigr)^{p-2}
\le
(1+a)^{p-2}.
$$
Thus
\begin{equation*}
\int_{S(0,1)}
\bigl(|e\cdot z_\sigma|+a\bigr)^{p-2}\,d\sigma
\le
\sigma(S(0,1))(1+a)^{p-2}.
\end{equation*}

It remains to consider $1<p<2$. Put $$\beta=2-p\in(0,1).$$ Then $p-2=-\beta.$
We need to estimate
\[
\int_{S(0,1)}
\bigl(|e\cdot z_\sigma|+a\bigr)^{-\beta}\,d\sigma .
\]

First suppose $0\le a\le 1$. Since $a\ge0$, we have
\[
|\cos\psi_1|\,|\omega_1|+a
\ge
|\cos\psi_1|\,|\omega_1|.
\]
Therefore,
\[
\bigl(|\cos\psi_1|\,|\omega_1|+a\bigr)^{-\beta}
\le
|\cos\psi_1|^{-\beta}|\omega_1|^{-\beta}.
\]
Consequently,
\[
\begin{aligned}
&\int_{S(0,1)}
\bigl(|e\cdot z_\sigma|+a\bigr)^{-\beta}\,d\sigma
\\
&\qquad\le
C(Q)
\int
\int_{\mathbb S^{2N-1}}
|\cos\psi_1|^{-\beta}
|\omega_1|^{-\beta}
|\cos\psi_1|^{2N-1}
\,d\omega\,d\psi_1
\\
&\qquad=
C(Q)
\left(
\int |\cos\psi_1|^{2N-1-\beta}\,d\psi_1
\right)
\left(
\int_{\mathbb S^{2N-1}}|\omega_1|^{-\beta}\,d\omega
\right).
\end{aligned}
\]

Both integrals are finite. Indeed, since $2N-1-\beta>-1$, we have
$$\int |\cos\psi_1|^{2N-1-\beta}\,d\psi_1<\infty.$$
Also, since \(\beta<1\),
\[
\int_{\mathbb S^{2N-1}}|\omega_1|^{-\beta}\,d\omega<\infty.
\]
Hence
\[
\int_{S(0,1)}
\bigl(|e\cdot z_\sigma|+a\bigr)^{-\beta}\,d\sigma
\le C(Q,p).
\]
Since \(0\le a\le 1\), we also have
\begin{equation*}
(1+a)^{-\beta}\sim 1.
\end{equation*}
Thus
\begin{gather*}
\int_{S(0,1)}
\bigl(|e\cdot z_\sigma|+a\bigr)^{p-2}\,d\sigma
\le
C(Q,p)(1+a)^{p-2}.
\end{gather*}

Finally suppose $a\ge1$. Then $|e\cdot z_\sigma|+a\ge a,$
and hence
\begin{gather*}
\bigl(|e\cdot z_\sigma|+a\bigr)^{-\beta}
\le
a^{-\beta}.
\end{gather*}
Therefore,
\[
\int_{S(0,1)}
\bigl(|e\cdot z_\sigma|+a\bigr)^{-\beta}\,d\sigma
\le
\sigma(S(0,1))a^{-\beta}
\le
C(Q)(1+a)^{-\beta}.
\]
Since \(-\beta=p-2\), this gives
\begin{gather*}
\int_{S(0,1)}
\bigl(|e\cdot z_\sigma|+a\bigr)^{p-2}\,d\sigma
\le
C(Q,p)(1+a)^{p-2}.
\end{gather*}

Combining the cases $p\ge2$ and $1<p<2$, we obtain
\begin{equation*}
    \int_{S(0,1)}
\bigl(|e\cdot z_\sigma|+a\bigr)^{p-2}\,d\sigma
\le
C(Q,p)(1+a)^{p-2}.
\end{equation*}
This proves the lemma.
\end{proof}
  \subsection{Principal Value}
    Our objective is to establish that the principal value associated with the operator $\mathcal{L}$ is well defined for sufficiently smooth functions. For this purpose, we obtain uniform estimates on small balls, which are stated in the two lemmas below. We now estimate the principal value integral at points where the horizontal gradient is nonzero.
    \begin{lemma}\label[lemma]{lem-pv1}(Principal value estimate away from critical points)
Let $B(\xi,\varepsilon) \subset \mathcal{D} \Subset \Omega$ and $u \in C^2(\mathcal{D})$. Assume either $p>\frac{2}{2-s}$ or $\mathcal{D} \Subset \{\xi : d_u(\xi)>0\}$. Then, we have the following:
\begin{itemize}
\item [(i)]
\begin{align*}
\Big| \textup{P.V.}\int_{B(\xi,\varepsilon)}|u(\xi)-u(\eta)|^{p-2} ( u(\xi)-u(\eta)) K(\xi,\eta) d\eta \Big| \leq C_\varepsilon,
\end{align*}
for some $C_\varepsilon>0$, ($C_\varepsilon$ is independent of $x$).
\item [(ii)] $C_\varepsilon \to 0$ as $\varepsilon \to 0$.
\end{itemize}
    \end{lemma}
    \begin{proof}
        If $\nabla_\hn u (\xi)=0$ and $p>\tfrac{2}{2-s}$, the result follows by the use of coarea formula on the Heisenberg group (\cite[Prop~5.4.4]{Bonfig2007}), $u\in C^2$, $\nabla_\hn u(\xi)=0$ and the stratified Taylor's formula on the Heisenberg group (\cite[Theorem~20.3.3]{Bonfig2007}), which implies
        \begin{align*}
            |u(\xi)-u(\eta)| \leq C\dhn{\eta}{\xi}^2
        \end{align*} for some constant $C$.
        Hence, it is enough to consider only the case $\nabla_\hn u\neq 0$.
        Let $l(\eta) = u(\xi) +\nabla_\hn(\xi)\cdot (z_\eta-z_\xi)$ be the horizontal affine part of $u$ near $\xi$. Let $g(t)=|t|^{p-2} t$, we have
        \begin{align*}
            &\Bigg|\int_{B(\xi,\varepsilon)}|u(\xi)-u(\eta)|^{p-2}(u(\xi)-u(\eta))K(\xi,\eta)\,d\eta\Bigg|\\
            &\qquad\leq \int_{B(\xi,\varepsilon)}|g(u(\xi)-u(\eta))-g(l(\xi)-l(\eta))|K(\xi,\eta)\,d\eta\\
            &\qquad\leq \int_{B(\xi,\varepsilon)}(|l(\xi)-l(\eta)|+|u(\eta)-l(\eta)|)^{p-2}|u(\eta)-l(\eta)|K(\xi,\eta)\,d\eta\\
            &\qquad = c\int_{B(\xi,\varepsilon)}(|\nabla_\hn u(\xi)\cdot (z_\eta-z_\xi)|+|u(\eta)-l(\eta)|)^{p-2}|u(\eta)-l(\eta)|K(\xi,\eta)\,d\eta,
        \end{align*} where, in the first step, we use \cref{lem:affine}, and in the second step, we use \cref{lem:algebraic}.

        Now, by left-translation, we have, with $\tau:=\sup_D\{X_iX_ju(\xi):1\leq i,j\leq 2N\}$,
        \begingroup
        \allowdisplaybreaks
        \begin{align}\label{eq:est1}
            &\Bigg|\int_{B(\xi,\varepsilon)}|u(\xi)-u(\eta)|^{p-2}(u(\xi)-u(\eta))K(\xi,\eta)\,d\eta\Bigg|\nonumber\\
            &\qquad \leq c\int_{B(0,\varepsilon)}\Big(|\nabla_\hn u(\xi)\cdot (z_{\xi\circ\zeta}-z_\xi)|+|u(\xi\circ\zeta)-l(\xi\circ\zeta)|\Big)^{p-2}\nonumber\\
            &\hspace{3cm}|u(\xi\circ\zeta)-l(\xi\circ\zeta)|K(\xi,\xi\circ\zeta)\,d\zeta\nonumber\\
            &\qquad \leq c\int_{B(0,\varepsilon)}(|\nabla_\hn u(\xi)\cdot (z_{\xi\circ\zeta}-z_\xi)|+\tau|\zeta|_\hn^2)^{p-2}\tau|\zeta|_\hn^2K(\xi,\xi\circ\zeta)\,d\zeta\nonumber\\
            &\qquad \leq c\int_{B(0,\varepsilon)}(|\nabla_\hn u(\xi)\cdot z_\zeta|+\tau|\zeta|_\hn^2)^{p-2}\tau|\zeta|_\hn^2|\zeta|_\hn^{-Q-sp}\,d\zeta\nonumber\\
            &\qquad \leq c\int_{B(0,\varepsilon)}(|\nabla_\hn u(\xi)\cdot z_\zeta|+\tau\|\zeta\|_\hn^2)^{p-2}\tau\|\zeta\|_\hn^2\|\zeta\|_\hn^{-Q-sp}\,d\zeta\nonumber\\
            &\qquad = c\int_0^\varepsilon\int_{S(0,1)}(|\nabla_\hn u(\xi)\cdot z_\sigma|r+\tau r^2)^{p-2}\tau r^2r^{-Q-sp+Q-1}\,dr\,d\sigma\nonumber\\
            &\qquad = c\tau\int_0^\varepsilon\int_{S(0,1)}\left(\frac{|\nabla_\hn u(\xi)\cdot z_\sigma|}{|\nabla_\hn u(\xi)|}+\frac{\tau r}{|\nabla_\hn u(\xi)|}\right)^{p-2}\nonumber\\
            &\hspace{4cm}|\nabla_\hn u(\xi)|^{p-2}\,d\sigma r^{p(1-s)}\frac{dr}{r}\nonumber\\
            &\qquad \leq c\tau\int_0^\varepsilon\left(1+\frac{\tau r}{|\nabla_\hn u(\xi)|}\right)^{p-2}|\nabla_\hn u(\xi)|^{p-2}\,r^{p(1-s)}\frac{dr}{r},
            \end{align} where in the second step, we have used the monotonicity of $(a+b)^{p-2}b$ with respect to $b$ when $a,b\geq 0$ as well as the Taylor's theorem; in the third step, we use the bounds on $K$; in the fourth step, we switch to another homogeneous norm $\|\cdot\|_{\hn}$; in the fifth step, we switch to polar coordinates as defined in \cref{eq:polar}; and in the last step, we use \cref{lem:sphere}.

            If $p\geq 2$, we have from \cref{eq:est1} that
            \begin{align*}
            &\Bigg|\int_{B(\xi,\varepsilon)}|u(\xi)-u(\eta)|^{p-2}(u(\xi)-u(\eta))K(\xi,\eta)\,d\eta\Bigg|\\
            &\qquad \leq c\tau\int_0^\varepsilon\left(1+\frac{\tau^{p-2} r^{p-2}}{|\nabla_\hn u(\xi)|^{p-2}}\right)|\nabla_\hn u(\xi)|^{p-2}\,r^{p(1-s)}\frac{dr}{r}\\
            &\qquad \leq c\tau \sup_{D}|\nabla_{\hn} u(\xi)|^{p-2}\varepsilon^{p(1-s)} + c\tau^{p-1}\varepsilon^{p-2+p(1-s)}.
            \end{align*}

            If $\frac{2}{2-s}<p< 2$, we have from \cref{eq:est1} that
            \begin{align*}
            &\Bigg|\int_{B(\xi,\varepsilon)}|u(\xi)-u(\eta)|^{p-2}(u(\xi)-u(\eta))K(\xi,\eta)\,d\eta\Bigg|\\
            &\qquad \leq c\tau\int_0^\varepsilon\left(\frac{\tau^{p-2} r^{p-2}}{|\nabla_\hn u(\xi)|^{p-2}}\right)|\nabla_\hn u(\xi)|^{p-2}\,r^{p(1-s)}\frac{dr}{r}\\
            &\qquad \leq c\tau^{p-1}\varepsilon^{p-2+p(1-s)}.
            \end{align*}

            Finally, if $1<p\leq \frac{2}{2-s}$ and $D\Subset\{\xi\in\hn : d_u(\xi)>0\}$, then $\inf_D|\nabla_\hn u|>0$ and we have from \cref{eq:est1} that
            \begin{align*}
            &\Bigg|\int_{B(\xi,\varepsilon)}|u(\xi)-u(\eta)|^{p-2}(u(\xi)-u(\eta))K(\xi,\eta)\,d\eta\Bigg|\\
            &\qquad \leq c\tau\int_0^\varepsilon|\nabla_\hn u(\xi)|^{p-2}\,r^{p(1-s)}\frac{dr}{r}\\
            &\qquad \leq c\tau \sup_{D}|\nabla_{\hn} u(\xi)|^{p-2}\varepsilon^{p(1-s)}.
            \end{align*}
            \endgroup
 In all three cases, we obtain the result.
    \end{proof}
    The next result complements the previous lemma by treating the more delicate case where the horizontal gradient vanishes.
   \begin{lemma}\label{lem-pv2} (Principal value estimate near critical points)
Let $1<p \leq \frac{2}{2-s}$, $\mathcal{D} \subset {\Omega}$ and $u \in C^2_\beta(\mathcal{D})$ for $\beta>\frac{sp}{p-1}$. Also, assume that $B_\varepsilon(\xi) \subset \mathcal{D}$ and $\xi \in \mathcal{D}$ such that $d_u(\xi)<\varepsilon<1$. Then we have the following:
\smallskip 
\noindent
\begin{enumerate}
\item [(i)] There is a $C_\varepsilon>0$ (independent of $\xi$) such that \begin{align*}
\Big|\textup{P.V.}\int_{B(\xi,\varepsilon)}|u(\xi)-u(\eta)|^{p-2} ( u(\xi)-u(\eta)) K(\xi,\eta) d\eta \Big| \leq C_\varepsilon.
\end{align*}
\smallskip \noindent
\item [(ii)] $C_\varepsilon \to 0$ as $\varepsilon \to 0$.
\end{enumerate}
   \end{lemma}
 \begin{proof}
If $\nabla_\hn u (\xi)=0$, the result follows by the use of coarea formula on the Heisenberg group (\cite[Prop~5.4.4]{Bonfig2007}), $u\in C^2_\beta$, $\nabla_\hn u(\xi)=0$ and the stratified Taylor's formula on the Heisenberg group (\cite[Theorem~20.3.3]{Bonfig2007}), which implies
        \begin{align*}
            |u(\xi)-u(\eta)| \leq C\dhn{\eta}{\xi}^2d_u(\xi)^{\beta-2}
        \end{align*} for some constant $C$.
        Hence, it is enough to consider only the case $\nabla_\hn u\neq 0$.
        Let $l(\eta) = u(\xi) +\nabla_\hn(\xi)\cdot (z_\eta-z_\xi)$ be the horizontal affine part of $u$ near $\xi$. Let $g(t)=|t|^{p-2} t$, we have
        \begingroup
        \allowdisplaybreaks
        \begin{align*}
            &\Bigg|\int_{B(\xi,\varepsilon)}|u(\xi)-u(\eta)|^{p-2}(u(\xi)-u(\eta))K(\xi,\eta)\,d\eta\Bigg|\\
            &\qquad\leq \int_{B(\xi,\varepsilon)}|g(u(\xi)-u(\eta))-g(l(\xi)-l(\eta))|K(\xi,\eta)\,d\eta\\
            &\qquad\leq \int_{B(\xi,\varepsilon)}(|l(\xi)-l(\eta)|+|u(\eta)-l(\eta)|)^{p-2}|u(\eta)-l(\eta)|K(\xi,\eta)\,d\eta\\
            &\qquad = c\int_{B(\xi,\varepsilon)}(|\nabla_\hn u(\xi)\cdot (z_\eta-z_\xi)|+|u(\eta)-l(\eta)|)^{p-2}|u(\eta)-l(\eta)|K(\xi,\eta)\,d\eta,
        \end{align*}\endgroup where, in the first step, we use \cref{lem:affine}, and in the second step, we use \cref{lem:algebraic}.

        Now, by left-translation, we have, with $\tau:=\sup_D\{X_iX_ju(\xi):1\leq i,j\leq 2N\}$,
        \begingroup
        \allowdisplaybreaks
        \begin{align}\label{eq:est2}
            &\Bigg|\int_{B(\xi,\varepsilon)}|u(\xi)-u(\eta)|^{p-2}(u(\xi)-u(\eta))K(\xi,\eta)\,d\eta\Bigg|\nonumber\\
            &\qquad \leq c\int_{B(0,\varepsilon)}\left(|\nabla_\hn u(\xi)\cdot (z_{\xi\circ\zeta}-z_\xi)|+|u(\xi\circ\zeta)-l(\xi\circ\zeta)|\right)^{p-2}\nonumber\\
            &\hspace{3cm}|u(\xi\circ\zeta)-l(\xi\circ\zeta)|K(\xi,\xi\circ\zeta)\,d\zeta\nonumber\\
            &\qquad \leq c\int_{B(0,\varepsilon)}(|\nabla_\hn u(\xi)\cdot (z_{\xi\circ\zeta}-z_\xi)|+\tau|\zeta|_\hn^2)^{p-2}\tau|\zeta|_\hn^2K(\xi,\xi\circ\zeta)\,d\zeta\nonumber\\
            &\qquad \leq c\int_{B(0,\varepsilon)}(|\nabla_\hn u(\xi)\cdot z_\zeta|+\tau|\zeta|_\hn^2)^{p-2}\tau|\zeta|_\hn^2|\zeta|_\hn^{-Q-sp}\,d\zeta\nonumber\\
            &\qquad \leq c\int_{B(0,\varepsilon)}(|\nabla_\hn u(\xi)\cdot z_\zeta|+\tau\|\zeta\|_\hn^2)^{p-2}\tau\|\zeta\|_\hn^2\|\zeta\|_\hn^{-Q-sp}\,d\zeta\nonumber\\
            &\qquad = c\int_0^\varepsilon\int_{S(0,1)}(|\nabla_\hn u(\xi)\cdot z_\sigma|r+\tau r^2)^{p-2}\tau r^2r^{-Q-sp+Q-1}\,dr\,d\sigma\nonumber\\
            &\qquad = c\int_0^\varepsilon\int_{S(0,1)}\left(\frac{|\nabla_\hn u(\xi)\cdot z_\sigma|}{|\nabla_\hn u(\xi)|}+\frac{\tau r}{|\nabla_\hn u(\xi)|}\right)^{p-2}\nonumber\\
            &\hspace{5cm}\tau|\nabla_\hn u(\xi)|^{p-2}\,d\sigma\,r^{p(1-s)}\frac{dr}{r}\nonumber\\
            &\qquad \leq c\int_0^\varepsilon\int_{S(0,1)}\left(\frac{|\nabla_\hn u(\xi)\cdot z_\sigma|}{|\nabla_\hn u(\xi)|}+\frac{(d_u(\xi)+r)^{\beta-2} r}{|\nabla_\hn u(\xi)|}\right)^{p-2}\nonumber\\
            &\hspace{4cm}(d_u(\xi)+r)^{\beta-2}|\nabla_\hn u(\xi)|^{p-2}\,d\sigma\,r^{p(1-s)}\frac{dr}{r}\nonumber\\
            &\qquad \leq c\int_0^\varepsilon\left(1+\frac{(d_u(\xi)+r)^{\beta-2} r}{|\nabla_\hn u(\xi)|}\right)^{p-2}(d_u(\xi)+r)^{\beta-2}|\nabla_\hn u(\xi)|^{p-2}\,r^{p(1-s)}\frac{dr}{r},
            \end{align}\endgroup where in the second step, we have used the monotonicity of $(a+b)^{p-2}b$ with respect to $b$ when $a,b\geq 0$ as well as the Taylor's theorem; in the third step, we use the bounds on $K$; in the fourth step, we switch to another homogeneous norm $\|\cdot\|_{\hn}$; in the fifth step, we switch to polar coordinates as defined in \cref{eq:polar}; in the seventh step, we use the definition of $C^2_{\beta}$; and in the last step, we use \cref{lem:sphere}.
            Now, we split the last integral into two parts:
            \begin{align*}
            &\Bigg|\int_{B(\xi,\varepsilon)}|u(\xi)-u(\eta)|^{p-2}(u(\xi)-u(\eta))K(\xi,\eta)\,d\eta\Bigg|\\
            & \leq c\int_0^{d_u(\xi)}\left(1+\frac{(d_u(\xi)+r)^{\beta-2} r}{|\nabla_\hn u(\xi)|}\right)^{p-2}(d_u(\xi)+r)^{\beta-2}|\nabla_\hn u(\xi)|^{p-2}\,r^{p(1-s)}\frac{dr}{r}\\
            &\quad+c\int_{d_u(\xi)}^\varepsilon\left(1+\frac{(d_u(\xi)+r)^{\beta-2} r}{|\nabla_\hn u(\xi)|}\right)^{p-2}(d_u(\xi)+r)^{\beta-2}|\nabla_\hn u(\xi)|^{p-2}\,r^{p(1-s)}\frac{dr}{r}\\
            &:=I_1+I_2
            \end{align*}
            The integral $I_1$ is estimated as
            \begingroup
        \allowdisplaybreaks
        \begin{align*}
                I_1&\leq c\int_0^{d_u(\xi)}d_u(\xi)^{\beta-2}|\nabla_\hn u(\xi)|^{p-2}\,r^{p(1-s)}\frac{dr}{r}\\
                &\leq cd_u(\xi)^{\beta-2}d_u(\xi)^{(\beta-2)(p-2)}d_u(\xi)^{p(1-s)}\\
                &\leq c\varepsilon^{\beta(p-1)-sp}
            \end{align*}
            \endgroup
            where we make use of the lower bound on $|\nabla_\hn u|$ in the definition of $C^2_\beta$. The exponent on $\varepsilon$ is positive since $\beta>\tfrac{sp}{p-1}$.
            For the integral $I_2$, we have
            \begin{align*}
                I_2&\leq c\int_{d_u(\xi)}^\varepsilon \left(\frac{r^{\beta-1}}{|\nabla_\hn u(\xi)|}\right)^{p-2} r^{\beta-2} |\nabla_\hn u(\xi)|^{p-2} r^{p(1-s)}\frac{dr}{r}\\
                &\leq c\int_{d_u(\xi)}^\varepsilon r^{\beta(p-1)-sp}\frac{dr}{r}\\
                &\leq c \varepsilon^{\beta(p-1)-sp}.
            \end{align*}
            The required estimate is found by combining the two estimates for $I_1$ and $I_2$.            
 \end{proof}
 \subsection{Continuity properties} Combining the previous estimates, we establish the continuity properties of the nonlocal operator under our structural assumptions.
 \begin{lemma}\label[lemma]{lem-cont1}(Continuity of the nonlocal operator)
Let $\xi_0 \in \Omega$ and $r>0$ be such that $B(\xi_0,r) \subset \Omega$ and $\phi \in L^{p-1}_{sp}(\mathbb{H}^N) \cap C^2(B_r(\xi_0))$. If $1<p \leq \frac{2}{2-s}$ and $\nabla_\hn u(\xi_0)=0$, further assume that $\phi \in C^2_{\beta}(B_r(\xi_0))$ for $\beta>\frac{sp}{p-1}$. Then, $\mathcal{L}\phi$ is continuous in $B_r(\xi_0)$.
 \end{lemma}
 \begin{proof}
    Let $\xi,\eta\in B(\xi_0,r)$. We will prove that for all $\varepsilon>0$, there is a $\delta>0$ such that whenever $\dhn{\xi}{\eta}< \delta$, we have $|\mathcal{L}\phi (\xi)-\mathcal{L}\phi (\eta)|<\varepsilon$. For this purpose, we will consider the local and the nonlocal parts of $\mathcal{L}$ separately.

    For the local part, in the case that $p>\tfrac{2}{2-s}$, given an $\varepsilon>0$ by \cref{lem-pv1}, there is a $\rho>0$ such that 
    \begin{align}\label{eq:est2.5}
        \Big|\textup{P.V.}\int_{B(\xi,\rho)}|u(\xi)-u(\zeta)|^{p-2} ( u(\xi)-u(\zeta)) K(\xi,\zeta) d\zeta \Big| < \varepsilon/4,\mbox{ and }\nonumber\\
        \Big|\textup{P.V.}\int_{B(\eta,\rho)}|u(\eta)-u(\zeta)|^{p-2} ( u(\eta)-u(\zeta)) K(\eta,\zeta) d\zeta \Big| < \varepsilon/4.
    \end{align}
    On the other hand, if $1<p\leq \tfrac{2}{2-s}$ and $\nabla_\hn u(\xi_0)=0$, there are two possibilities, either $\nabla_\hn u(\xi)\neq 0$ or $\nabla_\hn u(\xi)=0$. In the first case, there is a $\delta>0$ such that $\nabla_\hn u(\eta)\neq 0$ whenever $\dhn{\xi}{\eta}< \delta$. Then, again, by \cref{lem-pv1}, we can choose a $\rho>0$ such that \cref{eq:est2.5} holds. If $\nabla_\hn u(\xi)=0$, then $d_\phi(\eta)<\rho$ when $\dhn{\xi}{\eta}< \rho$ and again \cref{eq:est2.5} holds by Lemma \ref{lem-pv2}.
    For the nonlocal part, we further restrict $\dhn{\xi}{\eta}<\rho/3$. Then we have the following chain of estimates
    \begin{align*}
        &\mathbbm{1}_{\hn\setminus B(\eta,\rho)}(\zeta)|\phi(\eta)-\phi(\zeta)|^{p-1}K(\eta,\zeta)\\
        &\qquad \leq c\mathbbm{1}_{\hn\setminus B(\eta,\rho)}(\zeta)(|\phi(\eta)|^{p-1}+|\phi(\zeta)|^{p-1})\dhn{\eta}{\zeta}^{-Q-sp}\\
        &\qquad \leq c\mathbbm{1}_{\hn\setminus B(\xi,2\rho/3)}(\zeta)\left(\sup_{B(\xi,\rho/3)}|\phi(\cdot)|^{p-1}+|\phi(\zeta)|^{p-1}\right)\dhn{\xi}{\zeta}^{-Q-sp}.
    \end{align*}
    Therefore, by an application of dominated convergence theorem along with continuity of $K$ away from the diagonal as well as $\phi\in L^{p-1}_{sp}$, we have
    \begin{align}
        \label{eq:est3}
        &\int_{\hn\setminus B(\eta,\rho)} |\phi(\eta)-\phi(\zeta)|^{p-2}(\phi(\eta)-\phi(\zeta))K(\eta,\zeta)\,d\zeta\nonumber\\
        &\qquad\rightarrow \int_{\hn\setminus B(\xi,\rho)} |\phi(\xi)-\phi(\zeta)|^{p-2}(\phi(\xi)-\phi(\zeta))K(\xi,\zeta)\,d\zeta
    \end{align} as $\dhn{\xi}{\eta}\to 0$.
    Combining \cref{eq:est2.5} and \cref{eq:est3}, we have the required continuity result.
 \end{proof}
 The following perturbation result will be used in the proof of the comparison principle.
 \begin{lemma}\label[lemma]{lem-bound} (Perturbation lemma)
Let $\xi_0 \in \Omega$ and $r>0$ be such that $B(\xi_0,r) \subset \Omega$. Let $\phi \in L^{p-1}_{sp}(\mathbb{H}^N) \cap C^2(B(\xi_0,r))$ be such that the condition 3(a) or 3(b) of Definition \ref{def-visc} is satisfied. Then, for every $\eps>0$ and $\rho'>0$, there exists $\theta'>0, \,\, \rho \in (0, \rho')$ and $\eta \in C^2_0(B(\xi_0,\tfrac{\rho}{2}))$ satisfying $0 \leq \eta \leq 1$ and $\eta(\xi_0)=1$ such that the function $\phi_\theta:=\phi+\theta \eta$ satisfies
$$\sup_{B(\xi_0,\rho)}|\mathcal{L} \phi -\mathcal{L} \phi_\theta|<\varepsilon \,\text{ for } 0 \leq \theta<\theta'. $$
 \end{lemma}
 \begin{proof}
    Let $\eps>0$ and $\rho'>0$. There are three cases to consider. If $p>\tfrac{2}{2-s}$, we obtain, by \cref{lem-pv1} for sufficiently small $\delta>0$:
    \begin{align}\label{eq:est3.5}
        \Big| \textup{P.V.} \int_{B(\xi,\varepsilon)}|\phi_\theta(\xi)-\phi_\theta(\eta)|^{p-2} ( \phi_\theta(\xi)-\phi_\theta(\eta)) K(\xi,\eta) d\eta \Big| <\frac\varepsilon 4.
    \end{align}
    On the other hand, if $\nabla_\hn\phi(\xi_0)\neq 0$, there is a $\rho\in(0,\rho')$ and $\tau>0$ such that $|\nabla_\hn\phi(\xi)|>\tau$ in $B(\xi_0,2\rho)$ by continuity. Now, choosing any $\eta\in C^2_0(B(\xi_0,\rho/2))$ satisfying $0\leq \eta\leq 1$ with $\eta(\xi_0)=1$, there is a sufficiently small $\theta''>0$ such that for all $0\leq \theta<\theta''$, we have $|\nabla_\hn\phi_\theta|>\tau/2$ in $B(\xi_0,2\rho)$. Thus, in this case, \cref{eq:est3.5} still holds by \cref{lem-pv1} for a possibly smaller choice of $\delta$. Finally, when $1<p\leq\tfrac2{2-s}$, $\nabla_\hn\phi(\xi_0)=0$, and $\xi_0$ is an isolated critical point of $\phi$, we can again choose a possibly smaller $\rho>0$ so that $|\nabla_\hn\phi|\neq 0$ in $B(\xi_0,3\rho)\setminus\{\xi_0\}$. Choose $\eta$ such that $\eta\in C^2_0(B(\xi_0,\tfrac\rho2))$,  $0\leq \eta\leq 1$, $\eta=1$ in $B(\xi_0,\tfrac\rho4)$ and $|D^2_{\hn}\eta|\leq M\,d_{\eta}^{\beta-2}$ for some constant $M>0$. 
    Once again, by choosing a sufficiently small $\theta''$, for $0\leq\theta<\theta''$, we can ensure $\nabla\phi_\theta\neq 0$ in $B(\xi_0,2\rho)\setminus\{\xi_0\}$ and therefore, $d_{\phi_\theta}=d_{\phi}$ in $B(\xi_0,\rho)$ for this range of $\theta$. Moreover, with a still possibly smaller choice of $\theta$, we may have $\tfrac12|\nabla\phi|\leq|\nabla\phi_\theta|\leq 2|\nabla\phi|$ in $B(\xi_0,\rho)$. Finally, we can estimate
    \begin{align*}
        |D^2_{\hn}\phi_\theta|\leq |D^2_{\hn}\phi|+\theta|D^2_{\hn}\eta|\leq \|\phi\|_{C^2_\beta(B(\xi_0,\rho))}d_\phi^{\beta-2} + \theta M d_\eta^{\beta-2} \leq cd^{\beta-2}_{\phi_\theta}
    \end{align*} in $B(\xi_0,\rho)$ for small $\theta$ since $d_\eta\leq d_\phi=d_{\phi_\theta}$ in $B(\xi_0,\rho)$. As a result of these observations and choices, $\phi_\theta\in C^2_\beta(B(\xi_0,\rho))$ and by an application of \cref{lem-pv2}, we can find $\delta\in (0,\rho)$ such that \cref{eq:est3.5} holds in this case.
    Let $\xi\in B(\xi_0,\rho)$. With $g(t)=|t|^{p-2}t$, we can estimate using \cref{eq:est3.5} and \cref{lem:algebraic} that
    \begin{align*}
        &|\mathcal{L}\phi(\xi)-\mathcal{L}\phi_\theta(\xi)|\\
        &\qquad\leq \frac\varepsilon2 + \int_{\hn\setminus B(\xi,\delta)} |g(\phi(\xi)-\phi(\eta))-g(\phi_\theta(\xi)-\phi_\theta(\eta))|K(\xi,\eta)\,d\eta\\
        &\qquad \leq \frac\varepsilon2 + c\int_{\hn\setminus B(\xi,\delta)} 2\theta\left(|\phi(\xi)-\phi(\eta)|+2\theta\right)^{p-2}\dhn{\xi}{\eta}^{-Q-sp}\,d\eta,
    \end{align*} where apart from the monotonicity of $(a+b)^{p-2}b$ with respect to $b$ for $a,b\geq 0$, we also make use of
    \begin{align*}
        |\phi(\xi)-\phi(\eta)-\phi_\theta(\xi)+\phi_\theta(\eta)|\leq |\phi(\xi)-\phi_\theta(\xi)|+|\phi(\eta)-\phi_\theta(\eta)|\leq 2\theta.
    \end{align*}
    Now, if $p\in (1,2)$, the estimate is completed simply, as follows
    \begin{align*}
        |\mathcal{L}\phi(\xi)-\mathcal{L}\phi_\theta(\xi)| &\leq \frac\varepsilon2 + c\theta^{p-1}\int_{\hn\setminus B(\xi,\delta)} \dhn{\xi}{\eta}^{-Q-sp}\,d\eta\\
        & \leq \frac\varepsilon2 + c\delta^{-sp}\theta^{p-1} <\varepsilon,
    \end{align*} when $\theta$ is sufficiently small.
    On the other hand, for $p\in [2,\infty)$, we obtain
    \begin{align*}
        &|\mathcal{L}\phi(\xi)-\mathcal{L}\phi_\theta(\xi)|\\
        &\qquad \leq \frac\varepsilon2 + c\int_{\hn\setminus B(\xi,\delta)} \theta\left(|\phi(\xi)|^{p-2}+|\phi(\eta)|^{p-2}+\theta^{p-2}\right)\dhn{\xi}{\eta}^{-Q-sp}\,d\eta\\
        &\qquad\leq \frac\varepsilon2 + c\delta^{-sp}\theta^{p-1} + c\delta^{-sp}\theta \sup_{B(\xi_0,\rho)}|\phi|^{p-2} + c\delta^{-sp}\theta\sup_{\zeta\in B(\xi_0,\rho)}\textup{Tail}(\phi;\zeta,\delta)^{p-2},
    \end{align*} where we used H\"older's inequality to estimate
    \begin{align*}
        &\int_{\hn\setminus B(\xi,\delta)} |\phi(\eta)|^{p-2}\dhn{\xi}{\eta}^{-Q-sp}\,d\eta\\
        &\qquad\leq\left(\int_{\hn\setminus B(\xi,\delta)} \frac{1}{\dhn{\xi}{\eta}^{Q+sp}}\,d\eta
\right)^{\frac{1}{p-1}}\left(\int_{\hn\setminus B(\xi,\delta)} \frac{|\phi(\eta)|^{p-1}}{\dhn{\xi}{\eta}^{Q+sp}}\,d\eta\right)^{\frac{p-2}{p-1}}\\
&\qquad \leq c\delta^{sp/(p-1)}\delta^{-sp(p-2)/(p-1)}\textup{Tail}(\phi;\xi,\delta)^{p-2}\\
&\qquad \leq c\delta^{-sp} \sup_{\zeta\in B(\xi_0,\rho)}\textup{Tail}(\phi;\xi,\delta)^{p-2}.
    \end{align*}
    Thus we get 
    \begin{align*}
                |\mathcal{L}\phi(\xi)-\mathcal{L}\phi_\theta(\xi)|<\varepsilon
    \end{align*} when $\theta$ is small. The claim in the lemma follows by taking supremum over $\xi\in B(\xi_0,\rho)$.
 \end{proof}
 We conclude the preliminary section by showing that viscosity supersolutions satisfy the corresponding pointwise inequality for suitable test functions.
 \begin{lemma}\label[lemma]{lem-cont2}(Pointwise supersolutions are weak supersolutions)
Let $\xi_0 \in \Omega$ and $r_0>0$ be such that $B_r(\xi_0) \subset \Omega$. Let $u \in L^{p-1}_{sp}(\mathbb{H}^N) \cap C^2(B_r(\xi_0))$ and if $1<p \leq \frac{2}{2-s}$ and $
\nabla_
\hn u(\xi_0)=0$, then we assume there exists $
\beta>
\tfrac{sp}{p-1}$ such that $u \in C^2_{\beta}(B_r(\xi_0))$. If $\mathcal{L} u(\xi) \geq (\leq) f(\xi)$ for all $\xi \in B_r(\xi_0)$, then $u$ is a continuous weak supersolution (subsolution) in $B_r(\xi_0)$.
 \end{lemma}
 \begin{proof}
 We prove the lemma for supersolutions. The case of subsolutions follows after necessary modifications. As $u \in C^2(B_r(\xi_0))$ so $u \in W^{s,p}_{\textup{loc}}(B_r(\xi_0))$. Let $\phi \in C_0^{\infty}(B_r(\xi_0))$ and $\phi \geq 0$ be a test function. As $\mathcal{L} u \geq f$, by definition of $\mathcal{L}$, we note that,
\begin{align}\label{eq:u-est1}
 2\displaystyle\int_{\mathbb{H}^N \setminus B_{\varepsilon}(\xi)} |u(\xi)-u(\eta)|^{p-2} (u(\xi)-u(\eta)) K(\xi,\eta) d\eta \geq f(\xi)-\delta_{\varepsilon}(\xi),
\end{align}
for all $\xi \in \textup{supp}(\phi)$ and for every $\varepsilon>0$, where $\delta_\varepsilon(\xi) \to 0$ uniformly as $\varepsilon \to 0$, by the uniform continuity of $\mathcal{L}$ on the compact support of $
\phi$ as proved in \cref{lem-cont1}. Multiplying (\ref{eq:u-est1}) by the test function $\phi$ and integrating over $\mathbb{H}^N$ we obtain,
\begin{align}\label{eq:u-est2}
&2\int_{\mathbb{H}^N} \int_{\mathbb{H}^N} \Big( 1-\mathbbm{1}_{B(\xi,\varepsilon)}(\eta)\Big)|u(\xi)-u(\eta)|^{p-2}(u(\xi)-u(\eta))\phi(\xi) K(\xi,\eta)\,d\xi\,d\eta \nonumber\\
&\qquad\geq \int_{\mathbb{H}^N} f(\xi)\phi(\xi)d\xi-\int_{\mathbb{H}^N} \delta_\varepsilon(\xi)\phi(\xi)d\xi.
\end{align}
Interchanging the roles of $\xi$ and $\eta$ on L.H.S. of (\ref{eq:u-est2}) and using the symmetry of $K$ we get,
\begin{align}\label{eq:u-est3}
&2\int_{\mathbb{H}^N} \int_{\mathbb{H}^N} \Big( 1-\mathbbm{1}_{B(\eta,\varepsilon)}(\xi)\Big)|u(\xi)-u(\eta)|^{p-2}(u(\eta)-u(\xi))\phi(\eta) K(\xi,\eta)\,d\xi\,d\eta \nonumber\\
&\qquad\geq \int_{\mathbb{H}^N} f(\xi)\phi(\xi)d\xi-\int_{\mathbb{H}^N} \delta_\varepsilon(\xi)\phi(\xi)d\xi.
\end{align}
Adding (\ref{eq:u-est2}) and (\ref{eq:u-est3}) and changing the order of integration on the RHS of (\ref{eq:u-est2}), we note that for $\varepsilon>0$,
\begin{align*}
&\int_{\mathbb{H}^N} \int_{\mathbb{H}^N \setminus B(\eta,\varepsilon)} |u(\xi)-u(\eta)|^{p-2} ( u(\xi)-u(\eta)) (\phi(\xi)-\phi(\eta)) K(\xi,\eta) d\xi\,d\eta \nonumber\\
&\qquad\geq \int_{\mathbb{H}^N} f(\xi)\phi(\xi)d\xi-\|\delta_\varepsilon \phi\|_{L^1(\mathbb{H}^N)}.
\end{align*}
Now, we would like to pass to the limit in this inequality as $\eps\to 0$ by dominated convergence theorem which requires us to bound the integral above as we do below.
Let $\textup{supp}(\phi) \subset B_\rho \Subset B(\xi_0,r)$. By using H\"older inequality, we notice that,
\begin{align*}
&\int_{\mathbb{H}^N} \int_{\mathbb{H}^N} |u(\xi)-u(\eta)|^{p-2} ( u(\xi)-u(\eta)) (\phi(\xi)-\phi(\eta)) K(\xi,\eta) d\xi\,d\eta\nonumber\\
&\qquad \leq \int_{\mathbb{H}^N} \int_{\mathbb{H}^N} |u(\xi)-u(\eta)|^{p-1} (\phi(\xi)-\phi(\eta)) K(\xi,\eta) d\xi\,d\eta \nonumber\\
&\qquad \leq c \int_{B_\rho}\int_{B_\rho} \frac{|u(\xi)-u(\eta)|^{p-1}(\phi(\xi)-\phi(\eta))}{| \eta^{-1}\circ\xi |_{\mathbb{H}^N}^{Q+sp}} d\xi d\eta\\
&\qquad\qquad\qquad+c\int_{\mathbb{H}^N \setminus B_\rho}\int_{\textup{supp}(\phi)} \frac{|u(\xi)-u(\eta)|^{p-1}(\phi(\xi))}{| \eta^{-1} \circ\xi|_{\mathbb{H}^N}^{Q+sp}} d\xi d\eta\nonumber\\
 &\qquad \leq c \left( \int_{B_{\rho}} \int_{B_{\rho}} \frac{| u(\xi)-u(\eta)|^p}{| \eta^{-1}\circ \xi|_{\mathbb{H}^N}^{Q+sp}} d\xi d\eta  \right)^{\frac{p-1}{p}}
 \left(\int_{B_{\rho}} \int_{B_{\rho}} \frac{|\phi(\xi)-\phi(\eta)|^p}{|  \eta^{-1}\circ\xi|_{\mathbb{H}^N}^{Q+sp}} d\xi d\eta  \right)^{\frac{1}{p}}\nonumber\\
 &\qquad \qquad \qquad+c \int_{\mathbb{H}^N \setminus B_\rho} \int_{\textup{supp}(\phi)} \frac{(|u(\xi)|^{p-1}+|u(\eta)|^{p-1})\phi(\xi)}{|  \eta^{-1}\circ\xi|_{\mathbb{H}^N}^{Q+sp}}d\xi \,d\eta \nonumber\\
 &\qquad \leq c \|u\|_{W^{s,p}(B_\rho)}^{p-1}\|\phi\|_{W^{s,p}(B_\rho)}+c\|\phi\|_{L^1(B_\rho)}\text{ Tail }(u;\eta,d)^{p-1}<\infty
\end{align*}
where $\eta \in \textup{supp}(\phi)$ and $d:=\text{ dist }(\eta, \partial B_{\rho})$. Letting $\eps \to 0$,  by dominated convergence theorem, we see that,
\begin{align}\label{eq:u-est4}
&\int_{\mathbb{H}^N} \int_{\mathbb{H}^N} |u(\xi)-u(\eta)|^{p-2} ( u(\xi)-u(\eta)) (\phi(\xi)-\phi(\eta)) K(\xi,\eta) d\xi\,d\eta\nonumber\\
&\qquad\geq \int_{\mathbb{H}^N} f(\xi)\phi(\xi)d\xi.
\end{align}
This finishes the proof.
\end{proof}

    \section{Comparison Principle}\label{sec:comp}
A comparison principle holds for the weak solutions of the fractional $p$-Laplace equations on the Heisenberg group. The proof is similar to the one in the Euclidean case (see \cite[Lemma 9]{Lind2014}), however we prove it here for completeness.
\begin{theorem}\label{thm:comp} (Comparison principle)
    Let $u$ and $v$ be two continuous functions belonging to $W_0^{s,p}(\hn)$ and let $\Omega\subset\hn$ be an open set. If $u\geq v$ in $\hn\setminus\Omega$ and
    \begin{align*}
        &\int_\hn\int_\hn |u(\xi)-u(\eta)|^{p-2}(u(\xi)-u(\eta))(\phi(\xi)-\phi(\eta))K(\xi,\eta)\,d\xi\,d\eta\\
        &\qquad\geq \int_\hn\int_\hn |v(\xi)-v(\eta)|^{p-2}(v(\xi)-v(\eta))(\phi(\xi)-\phi(\eta))K(\xi,\eta)\,d\xi\,d\eta
    \end{align*} for all non-negative $\phi\in C_0(\Omega)$. Then, $u\geq v$ in $\Omega$.
\end{theorem}

\begin{proof}
    We rewrite the inequality as 
    \begin{align*}
        &\int_\hn\int_\hn \left(|u(\xi)-u(\eta)|^{p-2}(u(\xi)-u(\eta))- |v(\xi)-v(\eta)|^{p-2}(v(\xi)-v(\eta)\right)\\
        &\qquad\qquad(\phi(\xi)-\phi(\eta))K(\xi,\eta)\,d\xi\,d\eta\geq 0
    \end{align*}for all non-negative $\phi\in C_0(\Omega)$. Our aim is to show that the integrand is non-positive for the choice of $\phi=(v-u)_+:=\max\{v-u,0\}$.
    Let $g(t):=|t|^{p-2}t$. We use the identity $g(b)-g(a)=(p-1)(b-a)\int_0^1|a+t(b-a)|^{p-2}\,dt$ to write 
    \begin{align*}
        g(u(\xi)-u(\eta))-g(v(\xi)-v(\eta))=(p-1)\{u(\xi)-v(\xi)-(u(\eta)-v(\eta))\}L(\xi,\eta),
    \end{align*} where
    \begin{align*}
        L(\xi,\eta):=\int_0^1|v(\xi)-v(\eta)+t((u(\xi)-u(\eta))-(v(\xi)-v(\eta)))|^{p-2}\,dt.
    \end{align*}
    Observe that $L(\xi,\eta)\geq 0$ and further, $L(\xi,\eta)=0$ if and only if $u(\xi)=u(\eta)$ and $v(\xi)=v(\eta)$.
    Notice that with $\phi=(v-u)_+$, define $\psi=v-u=(v-u)_+-(v-u)_-$ and hence $\phi = \psi_+$. With this notation, the integrand becomes the factor $(p-1)L(\xi,\eta)K(\xi,\eta)$ multiplied with
    \begin{align*}
        &(\psi(\eta)-\psi(\xi))(\phi(\xi)-\phi(\eta))\\
        &\qquad=(\psi_+(\eta)-\psi_-(\eta)-\psi_+(\xi)+\psi_-(\xi))(\psi_+(\xi)-\psi_+(\eta))\\
        &\qquad=-(\psi_+(\xi)-\psi_+(\eta))^2+(\psi_-(\xi)-\psi_-(\eta))(\psi_+(\xi)-\psi_+(\eta))\\
        &\qquad=-(\psi_+(\xi)-\psi_+(\eta))^2-\psi_-(\xi)\psi_+(\eta)-\psi_-(\eta)\psi_+(\xi),
    \end{align*} where we used $\psi_+(\xi)\psi_-(\xi)=0$.
 It is a necessary condition, therefore, that $\psi_+(\xi)=\psi_+(\eta)$ or $L(\xi,\eta)=0$. The latter condition already implies the former. Therefore, it must hold that
\begin{align*}
    (v(\xi)-u(\xi))_+=(v(\eta)-u(\eta))_+
\end{align*} for almost every $\xi,\eta$. This gives us that $v(\xi) - u(\xi)=C\geq 0$ for some constant $C$ on the set where $v\geq u$. The boundary condition requires that $C=0$. It follows that $u\geq v$ almost everywhere.
\end{proof}

    \section{Proof of \cref{maintheorem1}}\label{sec:thm1}

    In this section, we prove one of our main theorems on weak solutions being viscosity solutions.
    \begin{proof}
        We prove the case of supersolutions. The case of subsolutions follows after necessary modifications. The function $u$ satisfies conditions $(1), (2),$ and $(4)$ of \cref{def-visc}. It remains to show $(3)$ from \cref{def-visc}. Consider a test function $\phi\in C^2(B(\xi_0,r))$ with $B(\xi_0,r)\subset\Omega$ such that $\phi(\xi_0)=u(\xi_0)$, $\phi\leq u$ in $B(\xi_0,r)$ and either $(a)$ or $(b)$ in \cref{def-visc}(3) holds. We need to show that $\mathcal{L}\phi_r(\xi_0)\geq f(\xi_0)$, where
        \begin{align*}
            \phi_r(\xi):=\begin{cases}
                \phi(\xi)&\mbox{ for }\xi\in B(\xi_0,r),\\
                u(\xi)&\mbox{ otherwise}.
            \end{cases}
        \end{align*}
        Instead suppose that $\mathcal{L}\phi_r(\xi_0)< f(\xi_0)$, then by \cref{lem-cont1} and the continuity of $f$, there is $r_1\in (0,r)$ and $\eps>0$ such that $\mathcal{L}\phi_r(\xi_0)< f(\xi_0)-\eps$ in $B(\xi_0,r_1)$. By \cref{lem-bound}, there is $r_2\in (0,r_1)$, $\hat\delta>0$, and a non-negative test function $\chi\in C^2_0(B(\xi_0,\tfrac{r_2}2))$ satisfying $\chi(\xi_0)=1$, $0\leq\chi\leq 1$, such that for $\phi_\delta:=\phi_r+\delta\chi$ we have
        \begin{align*}
            \sup_{B(\xi_0,r_2)}|\mathcal{L}\phi_r-\mathcal{L}\phi_\delta|<\tfrac\eps{2}
        \end{align*} for $\delta\in (0,\hat\delta)$. In the range $1<p\leq \frac{2}{2-s}$, we have $\phi_\delta\in C^2_\beta(B(\xi_0,r_2))$ for some $\beta>\tfrac{sp}{p-1}$. Therefore, $\mathcal{L}\phi_\delta<\mathcal{L}\phi_r + \tfrac\eps{2}<f-\tfrac\eps{2}$ in $B(\xi_0,r_2)$. By \cref{lem-cont2}, we have that $\phi_\delta$ is a weak subsolution in $B(\xi_0,r_2)$. Since $\phi_\delta=\phi_r\leq u$ in $\hn\setminus B(\xi_0,\tfrac{r_2}{2})$, by the comparison principle in \cref{thm:comp}, we have $u\geq \phi_r+\delta\chi$ in $B(\xi_0,\tfrac{r_2}{2})$. In particular, this implies that 
        \begin{align*}
            u(\xi_0)\geq \phi_r(\xi_0)+\delta\chi(\xi_0) = \phi_r(\xi_0)+\delta > \phi_r(\xi_0) = \phi(\xi_0),
        \end{align*} which contradicts $\phi(\xi_0)=u(\xi_0)$, proving the theorem.
    \end{proof} 

    \section{Proof of \cref{maintheorem2}}\label{sec:thm2}

\subsection{Modified infimal convolution}: 
To exploit the left-translation invariance of the kernel, we introduce the following modified infimal convolution and establish its basic properties. These results are adapted from the works in \cite{Wang2005,Wang2007,JuJuu2012,Katz2015}.
\begin{lemma}\label{lem:inf-conv} (Properties of the modified infimal convolution)
Let $u:\mathbb H^N\to\mathbb R$ be bounded and lower semicontinuous, and let
$\varepsilon>0$. Define
\begin{equation}\label{eq:conv}
  u_\varepsilon(\xi):=\inf_{\eta\in\mathbb H^N}\left\{u(\eta)+
  \frac{|\eta\circ\xi^{-1}|_{\mathbb H^N}^{q}}{q\varepsilon^{q-1}}\right\}.  
\end{equation}
Then the following properties hold.

\noindent
(i) There exists $r(\varepsilon)>0$, with $r(\varepsilon)\to0$ as
$\varepsilon\to0$, such that
$$u_\varepsilon(\xi)=\inf_{\eta\in B_\varepsilon(\xi)}\left\{u(\eta)+\frac{|\eta\circ\xi^{-1}|_{\mathbb H^N}^{q}}{q\varepsilon^{q-1}}\right\},
$$
where
$B_\varepsilon(\xi):=\left\{\eta\in\mathbb H^N:|\eta\circ\xi^{-1}|_{\mathbb H^N}<r(\varepsilon)
\right\}
=B(0,r(\varepsilon))\circ\xi.$

\noindent
(ii) The family $u_\varepsilon$ is increasing as $\varepsilon\downarrow0$, and
$$u_\varepsilon(\xi)\to u(\xi) \qquad\text{for every }\xi\in\mathbb H^N .$$

\noindent
(iii) The function $u_\varepsilon$ is locally semiconcave in the Euclidean sense.
In particular, for every compact set $K\Subset\mathbb H^N$, there exists a constant
$C=C(K,\varepsilon,q,\operatorname{osc}u)>0$ such that
$$D^2 u_\varepsilon(\xi)\le C I\qquad\text{for a.e. }\xi\in K.$$

\noindent
(iv) The function $u_\varepsilon$ is locally Lipschitz continuous and twice
differentiable a.e. in the Euclidean sense.

\noindent
(v) The set
$$Y_\varepsilon(\xi):=\left\{\eta\in B_\varepsilon(\xi):u_\varepsilon(\xi)=u(\eta)+\frac{|\eta\circ\xi^{-1}|_{\mathbb H^N}^{q}}
{q\varepsilon^{q-1}}
\right\}$$
is non-empty and closed for every $\xi\in\mathbb H^N$.
\end{lemma}

\begin{proof}
Let
$$M:=\sup_{\mathbb H^N}u,\qquad m:=\inf_{\mathbb H^N}u. $$
Choose $r(\varepsilon)>0$ so that
$\frac{r(\varepsilon)^q}{q\varepsilon^{q-1}}=M-m.$
Then $r(\varepsilon)\to0$ as $\varepsilon\to0$. If $|\eta\circ\xi^{-1}|_{\mathbb H^N}\ge r(\varepsilon),$
then
\[
u(\eta)+\frac{|\eta\circ\xi^{-1}|_{\mathbb H^N}^{q}}{q\varepsilon^{q-1}}\ge m+\frac{r(\varepsilon)^q}{q\varepsilon^{q-1}}
=M \ge u(\xi) \ge u_\varepsilon(\xi).
\]
Thus the infimum may be restricted to $B_\varepsilon(\xi)$, proving (i). Since
$u_\varepsilon(\xi)\le u(\xi)$, by choosing $\eta=\xi$, we have
$$\limsup_{\varepsilon\to0}u_\varepsilon(\xi)\le u(\xi).$$
Moreover, if $0<\varepsilon<\varepsilon_0$, then
$\frac{1}{\varepsilon^{q-1}}>\frac{1}{\varepsilon_0^{q-1}},$
and hence
$u_\varepsilon(\xi)\ge u_{\varepsilon_0}(\xi).$
Thus $u_\varepsilon$ is increasing as $\varepsilon\downarrow0$. Let
$\eta_\varepsilon\in Y_\varepsilon(\xi)$. By (i),
$$|\eta_\varepsilon\circ\xi^{-1}|_{\mathbb H^N}<r(\varepsilon),$$
and therefore $\eta_\varepsilon\to\xi$ as $\varepsilon\to0$. By lower
semicontinuity of $u$,
\[
\liminf_{\varepsilon\to0}u_\varepsilon(\xi)=\liminf_{\varepsilon\to0}
\left(u(\eta_\varepsilon)+\frac{|\eta_\varepsilon\circ\xi^{-1}|_{\mathbb H^N}^{q}}
{q\varepsilon^{q-1}}\right)
\ge \liminf_{\varepsilon\to0}u(\eta_\varepsilon) \ge u(\xi).
\]
Together with the previous $\limsup$ estimate, this proves (ii).

We next prove semiconcavity. Fix a compact set $K\Subset\mathbb H^N$. By (i), for
$\xi\in K$, the relevant points $\eta$ in the infimum satisfy
$\eta\in B_\varepsilon(\xi).$
Hence $(\eta,\xi)$ ranges in a compact subset of $\mathbb H^N\times\mathbb H^N$. For fixed $\eta$, define
\[
\Psi_\eta(\xi)
:=
u(\eta)
+
\frac{|\eta\circ\xi^{-1}|_{\mathbb H^N}^{q}}
{q\varepsilon^{q-1}}.
\]
Since $q\ge4$, the map $\xi\mapsto |\eta\circ\xi^{-1}|_{\mathbb H^N}^{q}$
is $C^2$ in the Euclidean variables. Therefore there exists a constant
$C=C(K,\varepsilon,q,\operatorname{osc}u)>0$, independent of $\eta$, such that
\[
D^2_\xi\Psi_\eta(\xi)\le C I
\qquad\text{for }\xi\in K,
\] where $D^2_\xi\Psi_\eta$ is the Euclidean Hessian.
Consequently, each function
\[
\xi\mapsto \Psi_\eta(\xi)-C|\xi|_E^2
\]
is Euclidean concave on $K$. Taking the infimum over $\eta$, we obtain that
$u_\varepsilon(\xi)-C|\xi|_E^2$
is Euclidean concave on $K$. Hence $u_\varepsilon$ is Euclidean locally semiconcave, and by Alexandrov's theorem, $u_\varepsilon$ is twice differentiable a.e. and
\[
D^2u_\varepsilon(\xi)\le C I
\qquad\text{for a.e. }\xi\in K.
\]
This proves (iii).

Since locally semiconcave functions are locally Lipschitz and twice differentiable almost everywhere by Alexandrov's theorem, (iv) follows.

Finally, we prove (v). By (i), the infimum is taken over the bounded set $B_\varepsilon(\xi)$. Let $(\eta_j)\subset B_\varepsilon(\xi)$ be a minimizing sequence. Since $B_\varepsilon(\xi)$ is relatively compact, there exists a subsequence, still denoted by $(\eta_j)$, and some $\eta\in \overline{B_\varepsilon(\xi)}$ such that
$\eta_j\to\eta.$
By the lower semicontinuity of $u$ and the continuity of the homogeneous norm,
\[
u_\varepsilon(\xi) \ge u(\eta) +\frac{|\eta\circ\xi^{-1}|_{\mathbb H^N}^{q}}{q\varepsilon^{q-1}}.
\]
The reverse inequality follows from the definition of $u_\varepsilon$. Thus
$\eta\in Y_\varepsilon(\xi)$, so $Y_\varepsilon(\xi)\neq\emptyset$. If $(\eta_j)\subset Y_\varepsilon(\xi)$ and $\eta_j\to\eta$, the same lower
semicontinuity argument gives
\[
u_\varepsilon(\xi) \ge u(\eta) + \frac{|\eta\circ\xi^{-1}|_{\mathbb H^N}^{q}}{q\varepsilon^{q-1}},
\]
while the reverse inequality again follows from the definition. Hence
$\eta\in Y_\varepsilon(\xi)$. Therefore $Y_\varepsilon(\xi)$ is closed.
This completes the proof.
\end{proof}
\begin{remark}
The intrinsic distance is chosen in the form \(|\eta\circ\xi^{-1}|_{\mathbb H^N}\),
which is compatible with the left-translation invariance of the kernel. This modification is essential in the noncommutative setting of the Heisenberg group and differs from the classical Euclidean infimal convolution.
\end{remark}

Now that we have the properties of the inf/sup-convolutions, the proof of \cref{maintheorem2} proceeds as follows. We first identify the equation satisfied by
$u_\varepsilon$ in the viscosity (and pointwise) sense. Next, by exploiting the higher
regularity and the specific properties of $u_\varepsilon$, we pass from the pointwise
formulation to the corresponding weak formulation. Finally, letting $\varepsilon \to 0$,
we obtain the weak problem satisfied by the limit function $u$. We begin by determining
the equation fulfilled by $u_\varepsilon$.

We next show that the modified infimal convolution preserves the viscosity supersolution property.
\begin{lemma}\label{lem:2} (Stability of viscosity supersolutions under infimal convolution)
Let $u:\mathbb{H}^N \to \mathbb{R}$ be bounded and lower semicontinuous and $f\in C(\Omega)$. Then, for $\tfrac{2}{2-s}<p<\infty$, if $u$ is a viscosity supersolution of $\mathcal{L}u=f  \, \in \Omega$, then $u_{\varepsilon}$ is a viscosity supersolution of 
$\mathcal{L}u=f_\eps  \, \in \Omega_{r(\varepsilon)}$,
where $$\Omega_{r(\varepsilon)}:=\Big\{\xi \in \Omega: \text{dist}(\xi, \partial \Omega)>r(\varepsilon) \Big\}.$$ Also, $f_\varepsilon(x)=\inf_{y \in B(\xi,{r_{\varepsilon}})}f(y)$. Furthermore, $\mathcal{L}u_\varepsilon(x) \geq f_\varepsilon(x)$ a.e. $x \in \Omega_{r(\varepsilon)}$.
\end{lemma}

\begin{proof}
We prove the result for supersolutions. The proof for subsolutions follows by
reversing the inequalities.

Recall that, with the modified infimal convolution,
$$
u_\varepsilon(\xi)
=
\inf_{\eta\in\mathbb H^N}
\left\{
u(\eta)
+
\frac{|\eta\circ\xi^{-1}|_{\mathbb H^N}^{q}}
{q\varepsilon^{q-1}}
\right\}.
$$
By writing $\eta=\zeta\circ\xi,$
we have
$$\eta\circ \xi^{-1}=\zeta\circ\xi\circ \xi^{-1}=\zeta.$$
Thus, after restricting to \(\zeta\in B(0,r(\varepsilon))\), we write
\begin{gather*}
u_\varepsilon(\xi)
=\inf_{\zeta\in B(0,r(\varepsilon))}
\left\{u(\zeta\circ\xi)+\frac{|\zeta|_{\mathbb H^N}^{q}}{q\varepsilon^{q-1}}\right\}.
\end{gather*}
For each $\zeta\in B(0,r(\varepsilon))$, define
\begin{equation*}
    \phi_\zeta(\xi)
:=
u(\zeta\circ\xi)
+
\frac{|\zeta|_{\mathbb H^N}^{q}}
{q\varepsilon^{q-1}},
\qquad \xi\in\mathbb H^N .
\end{equation*}
We first prove that $\phi_\zeta$ is a viscosity supersolution of
\begin{gather*}
\mathcal{L}v=f_\varepsilon \quad\text{in }\Omega_{r(\varepsilon)},
\end{gather*}
where
$$
f_\varepsilon(\xi)
:=
\inf_{\zeta\in B(0,r(\varepsilon))} f(\zeta\circ\xi).
$$
Let $\xi_0\in\Omega_{r(\varepsilon)}$, and let $\phi\in C^2(B_r(\xi_0))\cap L^{p-1}_{sp}(\mathbb H^N)$
be an admissible test function such that
$$
\phi(\xi_0)=\phi_\zeta(\xi_0),
\qquad
\phi\le \phi_\zeta
\quad\text{in }\mathbb H^N.
$$
Set
\begin{equation*}
    \gamma=\zeta\circ\xi,\qquad \gamma_0=\zeta\circ\xi_0.
\end{equation*}
Since $\xi_0\in\Omega_{r(\varepsilon)}$ and \(\zeta\in B(0,r(\varepsilon))\), we have \(\gamma_0\in\Omega\).
Define
\begin{gather*}
\widetilde\phi(\gamma):=\phi(\zeta^{-1}\circ\gamma)-\frac{|\zeta|_{\mathbb H^N}^{q}}{q\varepsilon^{q-1}}.
\end{gather*}
Then $\widetilde\phi\in C^2$ in a neighbourhood of $\gamma_0$, and $\widetilde\phi\in L^{p-1}_{sp}(\mathbb H^N)$, because left translations preserve the Haar measure and are compatible with the homogeneous structure.
We now verify that $\widetilde\phi$ touches $u$ from below at $\gamma_0$. Indeed,
\begin{gather*}
\widetilde\phi(\gamma_0)=\phi(\zeta^{-1}\circ\gamma_0)-\frac{|\zeta|_{\mathbb H^N}^{q}}{q\varepsilon^{q-1}}
=\phi(\xi_0)-\frac{|\zeta|_{\mathbb H^N}^{q}}{q\varepsilon^{q-1}}.
\end{gather*}
Since \(\phi(\xi_0)=\phi_\zeta(\xi_0)\), we get
$$
\widetilde\phi(\gamma_0)=u(\zeta\circ\xi_0)=u(\gamma_0).$$
Moreover, for any \(\gamma\in\mathbb H^N\), writing
\[
\xi=\zeta^{-1}\circ\gamma,
\]
we obtain
\[
\widetilde\phi(\gamma)=\phi(\xi)-\frac{|\zeta|_{\mathbb H^N}^{q}}{q\varepsilon^{q-1}}\le\phi_\zeta(\xi)
-\frac{|\zeta|_{\mathbb H^N}^{q}}{q\varepsilon^{q-1}}=u(\zeta\circ\xi)=u(\gamma).
\]
Therefore, \(\widetilde\phi\) is an admissible test function for \(u\) at \(\gamma_0\).
Since \(u\) is a viscosity supersolution of \(\mathcal Lu=f\) in \(\Omega\), we have
\[
\mathcal{L}\widetilde\phi(\gamma_0)\ge f(\gamma_0).
\]
We claim that
$\mathcal{L}\widetilde\phi(\gamma_0)=\mathcal L\phi(\xi_0).$
Indeed, using the definition of \(\mathcal L\), the change of variables
$\gamma=\zeta\circ\eta,$
the invariance of Haar measure under left translations, and the left-translation
invariance of $K$, we obtain
\begin{equation*}
    \begin{aligned}
\mathcal{L}\widetilde\phi(\gamma_0)
&=
\operatorname{P.V.}\int_{\mathbb H^N}
|\widetilde\phi(\gamma_0)-\widetilde\phi(\gamma)|^{p-2}
\bigl(\widetilde\phi(\gamma_0)-\widetilde\phi(\gamma)\bigr)
K(\gamma_0,\gamma)\,d\gamma                                      \\
&=
\operatorname{P.V.}\int_{\mathbb H^N}
|\phi(\xi_0)-\phi(\eta)|^{p-2}
\bigl(\phi(\xi_0)-\phi(\eta)\bigr)
K(\zeta\circ\xi_0,\zeta\circ\eta)\,d\eta                         \\
&=
\operatorname{P.V.}\int_{\mathbb H^N}
|\phi(\xi_0)-\phi(\eta)|^{p-2}
\bigl(\phi(\xi_0)-\phi(\eta)\bigr)
K(\xi_0,\eta)\,d\eta                                               \\
&=
\mathcal L\phi(\xi_0).
\end{aligned}
\end{equation*}
Consequently,
$\mathcal{L}\phi(\xi_0)=\mathcal{L}\widetilde\phi(\gamma_0)\ge f(\gamma_0)=f(\zeta\circ\xi_0)\ge f_\varepsilon(\xi_0).$
Hence $\phi_\zeta$ is a viscosity supersolution of $\mathcal{L}v=f_\varepsilon$ in $\Omega_{r(\varepsilon)}$. We now prove that $u_\varepsilon$ itself is a viscosity supersolution. Let $\psi$ be an admissible test function touching $u_\varepsilon$ from below at
$\xi_0\in\Omega_{r(\varepsilon)}$, namely
\begin{gather*}
\psi(\xi_0)=u_\varepsilon(\xi_0),
\qquad
\psi\le u_\varepsilon
\quad\text{in }\mathbb H^N.
\end{gather*}
By the compactness property of the minimizing set in the infimal convolution, there
exists \(\widehat\zeta\in B(0,r(\varepsilon))\) such that
$$u_\varepsilon(\xi_0)=\phi_{\widehat\zeta}(\xi_0).$$
Since $\psi\le u_\varepsilon\le \phi_{\widehat\zeta}\quad\text{in }\mathbb H^N,$
we may use $\psi$ as a test function for the viscosity supersolution $\phi_{\widehat\zeta}$. Therefore
$\mathcal{L}\psi(\xi_0)\ge f_\varepsilon(\xi_0),$ which proves that $u_\varepsilon$ is a viscosity supersolution of
$$\mathcal{L}v=f_\varepsilon\quad\text{in }\Omega_{r(\varepsilon)}.$$
It remains to prove the almost everywhere pointwise inequality. Let
$\xi\in\Omega_{r(\varepsilon)}$ be a point where $u_\varepsilon$ is twice differentiable, and let $r>0$ be such that
$B_r(\xi)\Subset\Omega_{r(\varepsilon)}.$
Since $p>\frac{2}{2-s}$, the principal value is well-defined for $C^2$-test functions.
For $\delta>0$, define the quadratic Heisenberg polynomial
$$P_\delta(\gamma):=u_\varepsilon(\xi)+\langle\nabla_{\mathbb H^N}u_\varepsilon(\xi),z_h\rangle+\partial_t u_\varepsilon(\xi)t_h+
\frac12\left\langle\bigl(D_{\mathbb H^N}^{2,*}u_\varepsilon(\xi)-\delta I\bigr)z_h,z_h\right\rangle,
$$
where
$h=\xi^{-1}\circ\gamma=(z_h,t_h).$
By the second-order Taylor expansion on $\mathbb H^N$, and after possibly taking
$r>0$ smaller, the function
$$P_{\delta,r}(\gamma):=
\begin{cases}
P_\delta(\gamma), & \gamma\in B_r(\xi),\\
u_\varepsilon(\gamma), & \gamma\in\mathbb H^N\setminus B_r(\xi),
\end{cases}
$$
satisfies
$P_{\delta,r}(\xi)=u_\varepsilon(\xi), P_{\delta,r}\le u_\varepsilon \quad\text{in }\mathbb H^N. $
Thus, $P_{\delta,r}$ is an admissible test function for $u_\varepsilon$ at $\xi$. Since $u_\varepsilon$ is a viscosity supersolution, we have
$$\mathcal{L}P_{\delta,r}(\xi)\ge f_\varepsilon(\xi).$$
Because $P_{\delta,r}=u_\varepsilon$ in $\mathbb H^N\setminus B_r(\xi)$, the nonlocal part outside $B_r(\xi)$ coincides
with that of $u_\varepsilon$. Moreover, by Lemma \ref{lem-pv1}, the principal value of the
local part over $B_r(\xi)$ tends to zero as $r\to0$. Hence,
\begin{align*}
&\int_{\mathbb H^N}
|u_\varepsilon(\xi)-u_\varepsilon(\eta)|^{p-2}
\bigl(u_\varepsilon(\xi)-u_\varepsilon(\eta)\bigr)
K(\xi,\eta)\,d\eta \\
&=\int_{\mathbb H^N}
|P_{\delta,r}(\xi)-P_{\delta,r}(\eta)|^{p-2}
\bigl(P_{\delta,r}(\xi)-P_{\delta,r}(\eta)\bigr)
K(\xi,\eta)\,d\eta \\
&\geq \text{P.V.}\int_{\mathbb H^N}
|P_{\delta,r}(\xi)-P_{\delta,r}(\eta)|^{p-2}
\bigl(P_{\delta,r}(\xi)-P_{\delta,r}(\eta)\bigr)
K(\xi,\eta)\,d\eta - o_r(1) \\
&\geq f_\varepsilon(\xi)-o_r(1).
\end{align*}
Passing to the limit as $r\to 0$, we get
$$\mathcal{L}u_\varepsilon(\xi)\ge f_\varepsilon(\xi) \quad\text{for a.e. }\xi\in\Omega_{r(\varepsilon)}.$$
This completes the proof.
\end{proof}

Since the proof of the weak formulation relies on smooth approximations, we first recall the convergence properties of mollification on the Heisenberg group.
\begin{lemma} \cite[Chapter~1, Sections A, B]{FollandStein1982} \label{lem:mollification-gradient}
Let \(U\subset \mathbb H^N\) be open and let
\(v\in L^\infty_{\mathrm{loc}}(U)\) be locally Lipschitz in \(U\). Let
\(\eta\in C_c^\infty(\mathbb H^N)\), \(\eta\ge0\), \(\int_{\mathbb H^N}\eta=1\), and set
\(\eta_\delta(\gamma)=\delta^{-Q}\eta(\Phi_{\frac1\delta}\gamma)\) where $\Phi_\cdot$ is the dilation on $\hn$. For \(\xi\in U_\delta\Subset U\), define
\[
v_\delta(\xi)
=
\int_{\mathbb H^N}\eta_\delta(\gamma)
v (\gamma^{-1} \circ \xi)\,d\gamma .
\]
Then $v_\delta\to v \,\text{locally uniformly in }U.$
Moreover, $ \nabla_{\mathbb H^N}v_\delta(\xi) \to \nabla_{\mathbb H^N}v(\xi) \,\text{for a.e. }\xi\in U.$
\end{lemma}

We are now in a position to derive the weak formulation satisfied by the regularized viscosity supersolutions.

\begin{lemma}\label{lem:weak-form-ueps} (From pointwise to weak formulation)
Let \(p>\frac{2}{2-s}\). Let
\(u_\varepsilon\in L^{p-1}_{sp}(\mathbb H^N)\cap L^\infty_{\mathrm{loc}}(\mathbb H^N)\)
be locally Lipschitz and semiconcave in \(\Omega_{r(\varepsilon)}\). Assume that
\[
\mathcal L u_\varepsilon(\xi)\ge f_\varepsilon(\xi)
\quad\text{for a.e. }\xi\in\Omega_{r(\varepsilon)}.
\]
Then, for every non-negative \(\Psi\in C^\infty_0(\Omega_{r(\varepsilon)})\),
\[
\iint_{Q_{\operatorname{supp}\Psi}}
|u_\varepsilon(\xi)-u_\varepsilon(\gamma)|^{p-2}
\bigl(u_\varepsilon(\xi)-u_\varepsilon(\gamma)\bigr)
\bigl(\Psi(\xi)-\Psi(\gamma)\bigr)
K(\xi,\gamma)\,d\gamma d\xi
\ge
\int_{\operatorname{supp}\Psi} f_\varepsilon(\xi)\Psi(\xi)\,d\xi,
\] where \begin{align*}
Q_{\textup{supp}(\Psi)}:= (\mathbb{H}^N \times \mathbb{H}^N) \setminus \Big(\mathbb{H}^N \setminus \textup{supp}(\Psi) ) \times ( \mathbb{H}^N \setminus \textup{supp}(\Psi)\Big).
\end{align*}
\end{lemma}
\begin{proof}
For $\varepsilon, \delta>0$, let us consider the function $u_{\varepsilon, \delta}:\mathbb{R}^N \to \mathbb{R}$ be defined as:
\begin{align}\label{eq:u-de}
    u_{\varepsilon, \delta}(\xi)=\begin{cases}
\int_{B_{\delta(0)}}\eta_\delta(\gamma)u_\varepsilon(\gamma^{-1}\circ\xi) \,d\gamma, \, &\xi \in \Omega_{r(\varepsilon)},\\
u_\varepsilon(\xi),\, &\xi \in \mathbb{H}^N \setminus \Omega_{r(\varepsilon)}.
    \end{cases}
\end{align}
with $\eta_\delta(\gamma)=\delta^{-Q}\,\eta\!\left(\Phi_{1/\delta}(\gamma)\right)$,
where $\Phi_r(x,y,t)=(rx,\,ry,\,r^2t)$, $\gamma \in \mathbb{H}^N$ and $\eta$ is the standard mollifier. We note that $\{ u_{\varepsilon, \delta}: \varepsilon,\delta>0\}$ is a family of smooth semi-concave functions and $u_{\varepsilon,\delta} \to u_\varepsilon$ a.e. $\xi \in \mathbb{H}^N$ as $\delta \to 0$.
As $u_{\varepsilon,\delta} \in C^2(\Omega_{r(\varepsilon)}) \cap L_{sp}^{p-1}(\mathbb{H}^N)$, we see that for all $\Psi \in C_0^{\infty}(\Omega_{r(\varepsilon)})$ with $\mathcal{K}=supp(\Psi)$,
\begin{equation*}
  \int \int_{Q_{\mathcal{K}}}|u_{\varepsilon,\delta}(\xi)-u_{\varepsilon,\delta}(\gamma)|^{p-2} (u_{\varepsilon,\delta}(\xi)-u_{\varepsilon,\delta}(\gamma)) ( \Psi(\xi)-\Psi(\gamma)) K(\xi,\gamma)\,d\xi d\gamma=\int_{\mathcal{K}} (\mathcal{L}u_{\varepsilon, \delta})(\xi) \Psi(\xi)\,d\xi.
\end{equation*}
\textbf{\underline{Claim I}}. We have the following:
\begin{align}\label{eq:a}
    &\lim_{\delta \to 0} \int \int_{Q_{\mathcal{K}}}|u_{\varepsilon,\delta}(\xi)-u_{\varepsilon,\delta}(\gamma)|^{p-2} (u_{\varepsilon,\delta}(\xi)-u_{\varepsilon,\delta}(\gamma)) ( \Psi(\xi)-\Psi(\gamma)) K(\xi,\gamma)\,d\xi d\gamma \nonumber\\
    &=\int \int_{Q_{\mathcal{K}}}|u_{\varepsilon,\delta}(\xi)-u_{\varepsilon,\delta}(\gamma)|^{p-2} (u_{\varepsilon}(\xi)-u_{\varepsilon}(\gamma)) ( \Psi(\xi)-\Psi(\gamma)) K(\xi,\gamma)\,d\xi d\gamma.
\end{align}
\textit{Proof of Claim I}. For $\delta>0$, let us define $\mathcal{F}_\delta:Q_{\mathcal{K}} \to \mathbb{R}$ as
\begin{equation*}
    \mathcal{F}_{\delta}(\xi, \gamma):=|u_{\varepsilon,\delta}(\xi)-u_{\varepsilon,\delta}(\gamma)|^{p-2} (u_{\varepsilon,\delta}(\xi)-u_{\varepsilon,\delta}(\gamma)) ( \Psi(\xi)-\Psi(\gamma)) K(\xi,\gamma).
\end{equation*}
Observing symmetry, we note that,
\begin{equation}\label{eq:b}
    \int \int_{Q_{\mathcal{K}}} \mathcal{F}_{\delta}(\xi, \gamma) d\xi d\gamma=2 \int \int_{\mathcal{K} \times (\mathbb{H}^N \setminus \mathcal{K})}\mathcal{F}_{\delta}(\xi, \gamma) d\xi d\gamma+ {\int \int}_{ \mathcal{K} \times \mathcal{K}} \mathcal{F}_{\delta}(\xi, \gamma) d\xi d\gamma.
\end{equation}
Let us denote $r:=dist( \mathcal{K}, \partial \Omega_{r(\varepsilon)})>0$ and $\mathcal{K}_{\frac{r}{2}}:=\{\gamma \in \Omega_{r(\varepsilon)}: dist(\gamma, \mathcal{K}) \leq \frac{r}{2}\}$.

 If $(\xi,\gamma) \in \mathcal{K} \times \mathcal{K}_{\frac{r}{2}}$, by Young's inequality we have,
 \begin{align}\label{eq:c}
|\mathcal{F}_{\delta}(\xi, \gamma)| &\leq \frac{ |u_{\varepsilon, \delta}(\xi)-u_{\varepsilon, \delta}(\gamma)|^{p-1} |\Psi(\xi)-\Psi(\gamma)|}{\dhn{\xi}{\gamma}^{Q+sp}}\nonumber\\
&\leq C \Bigg( \frac{|u_{\varepsilon,\delta}(\xi)-u_{\varepsilon,\delta}(\gamma)|^p}{\dhn{\xi}{\gamma}^{Q+sp}}  + \frac{|\Psi(\xi)-\Psi(\gamma)|^p}{\dhn{\xi}{\gamma}^{Q+sp}}  \Bigg)
 \end{align}
where $C>0$ is independent of $\delta$. As $u_{\varepsilon}$ is locally Lipschitz, with $\delta>0$ small enough, we note that,
\begin{align*}
|u_{\varepsilon, \delta}(\xi)-u_{\varepsilon, \delta}(\gamma)|^{p} & = \left| \int_{B_{\delta}(0)} \Big( u_\varepsilon(\zeta^{-1}\circ\xi)-u_\varepsilon(\zeta^{-1}\circ\gamma)\Big) \eta_\delta(\zeta)d\zeta\right|^p \nonumber\\
&\leq C |\gamma^{-1}\circ\xi|^p,
\end{align*}
where $C>0$ depends on the Lipschitz constant of $u_\varepsilon$ in $\mathcal{K}_{\frac{3r}{4}}$ but independent of $\delta$ as $\| \eta_{\delta} \|_{L^1}=1.$ Hence, from (\ref{eq:c}), we conclude that,
\begin{align}\label{eq:cd}
    |\mathcal{F}_{\delta}(\xi, \gamma)|\leq C|\gamma^{-1}\circ\xi|^{p-Q-sp}\in L^1(\mathcal{K}\times\mathcal{K}_{\frac r2}),
\end{align} with $C$ independent of $\delta$.

In the case, when $|\gamma^{-1}\circ\xi|_\hn\geq \frac r2$, the kernel is not singular, thus using that $u_{\varepsilon,\delta}$ is locally bounded, we have
\begin{equation}\label{eq:d}
    |\mathcal{F}_{\delta}(\xi, \gamma)| \leq C \frac{1+|u_{\varepsilon}(\gamma)|^{p-1}}{|\gamma^{-1}\circ\xi|^{Q+sp}_{\mathbb{H}^N}} \in L^1( \mathcal{K} \times ( \mathbb{H}^N \setminus \mathcal{K}_{\frac{r}{2}} )),
\end{equation}
as $u_\varepsilon \in L^{p-1}_{sp}(\mathbb{H}^N)$, where $C>0$ is independent of $\delta$.  
By (\ref{eq:cd}) and (\ref{eq:d}), we can say that,
\begin{equation*}
    \mathcal{F}_{\delta}(\xi, \gamma) \leq C \mathcal{F}(\xi, \gamma),
\end{equation*}
where $C>0$ is independent of $\delta$ and $\mathcal{F} \in L^1(\mathcal{K} \times \mathbb{H}^N)$. Using Dominated Convergence theorem, from (\ref{eq:b}), Claim I follows.

\textbf{\underline{Claim II}}. We assert the following:
\begin{equation*}
\liminf_{\delta \to 0} \int_{\mathcal{K}} \mathcal{L}u_{\varepsilon, \delta}(\xi)\Psi(\xi)d\xi \geq 
\int_{\mathcal{K}} \liminf_{\delta \to 0} \mathcal{L}u_{\varepsilon, \delta}(\xi)\Psi(\xi)d\xi.
\end{equation*}

\textit{Proof of Claim II}. Let $0<{r}<dist(\mathcal{K}, \partial \Omega_{r(\varepsilon)})$. Then, $B_{{r}}(\xi) \subset \Omega_{r(\varepsilon)}$ and for $\gamma \in \mathbb{H}^N \setminus B_{{r}}(\xi)$, we have,
$|\gamma^{-1}\circ\xi|_{\mathbb{H}^N}> {r}.$ As $u_\varepsilon$ is locally bounded, analogous to the proof of Claim I, 
\begin{align}\label{eq:e}
&\Big|\int_{\mathbb{H}^N \setminus B_{{r}}(\xi)}
|u_{\varepsilon,\delta}(\xi)-u_{\varepsilon,\delta}(\gamma)|^{p-2}
(u_{\varepsilon,\delta}(\xi)-u_{\varepsilon,\delta}(\gamma)) K(\xi,\gamma)
 d\gamma \Big|\nonumber\\
&\leq \qquad C \int_{\mathbb{H}^N \setminus B_{r}(\xi)} \frac{1+|u_\varepsilon(\gamma)|^{p-1}}{|\xi \circ \gamma^{-1}|_{\mathbb{H}^N}^{Q+sp}} \, d\gamma
\end{align}
As $u_\varepsilon \in L^{p-1}_{sp}(\mathbb{H}^N)$, the right hand side of (\ref{eq:e}) is uniformly bounded with respect to the parameter $\delta>0$. 

Define
\begin{equation}\label{eq:f}
    \mathcal{I}_{B_{{r}(\xi)}}:=P.V. \int_{B_{{r}(\xi)}}|u_{\varepsilon,\delta}(\xi)-u_{\varepsilon,\delta}(\gamma)|^{p-2}
(u_{\varepsilon,\delta}(\xi)-u_{\varepsilon,\delta}(\gamma))K(\xi,\gamma) \, d\gamma.
\end{equation}

By definition and Lemma \ref{lem:inf-conv}(iii), we note that the Euclidean Hessian satisfies
\begin{equation}\label{eq:f-}
    D^{2}_{\xi}u_{\varepsilon, \delta} \leq C I \text{ in } \Omega_{r(\varepsilon)},
\end{equation}
where $D^2_\xi$ is the Euclidean Hessian of $u_{\varepsilon,\delta}$ and $C>0$ is independent of $\delta$. Therefore, by \cite[Lemma~3]{BieskeMan2005}, we have that the group Hessian satisfies
\begin{equation}\label{eq:f--}
    \nabla^{2,*}_{\hn}u_{\varepsilon, \delta} = \mathbb{A}D^{2}_{\xi}u_{\varepsilon, \delta} \mathbb{A}^T\overset{\ref{eq:f-}}{\leq} \tilde{C} I \text{ in } \Omega_{r(\varepsilon)},
\end{equation} where
\begin{align*}
    \mathbb{A}:=\begin{bmatrix}
1 & 0 & 0 & \cdots & 0 & 0 & 2y_1 \\
0 & 1 & 0 & \cdots & 0 & 0 & \vdots \\
0 & 0 & 1 & 0 & 0 & 0 & 2y_n \\
\vdots & \ddots & \ddots & 1 & 0 & 0 & -2x_1 \\
\vdots & \ddots & \ddots & \ddots & \ddots & \vdots & \vdots \\
0 & 0 & 0 & 0 & \cdots & 1 & -2x_n \\
0 & 0 & 0 & 0 & \cdots & 0 & 1 
\end{bmatrix} 
\end{align*}

Let \(h_\gamma:=\xi^{-1}\circ\gamma=(z_\gamma,t_\gamma)\). Now, let us recall the Taylor's theorem with an integral remainder~\cite[Prop.20.3.14]{Bonfig2007}.
\begin{align}\label{eq:tayl}
    u_{\varepsilon,\delta}(\gamma) = u_{\varepsilon,\delta}(\xi) + (\nabla_{\mathbb{H}^N} u_{\varepsilon, \delta}(\xi), \partial_t u_{\varepsilon})\cdot h_\gamma + \left\langle \int_0^1 \nabla^{2,*}_\hn u_{\varepsilon,\delta}(\xi\circ(s h_\gamma))(1-s)\,ds\,h_\gamma, h_\gamma\right\rangle.
\end{align}

Let \(L(t)=|t|^{p-2}t\). Let us define $$a:=u_{\varepsilon, \delta}(\xi)-u_{\varepsilon, \delta}(\gamma), b:=(\nabla_{\mathbb{H}^N} u_{\varepsilon, \delta}(\xi), \partial_t u_{\varepsilon, \delta}(\xi)) \cdot (\xi^{-1}\circ\gamma), $$

By Lemma \ref{lem:affine}, the principal value of the affine part vanishes. Hence, by Lemma \ref{lem:algebraic}, we have
\begingroup
\allowdisplaybreaks
\begin{align}\label{eq:h}
\mathcal I_{B_r(\xi)}
&=
\int_{B_r(\xi)}
(L(a)-L(b)) K(\xi,\gamma)
\,d\gamma\nonumber\\
& = (p-1)\int_{B_r(\xi)}(a-b)\left(\int_0^1|ta+(1-t)b|^{p-2}\,dt\right)\,d\gamma\nonumber\\
& \overset{\ref{eq:tayl}}{=} (p-1)\int_{B_r(\xi)}\left\langle -\int_0^1\nabla^{2,*}_\hn u_{\varepsilon,\delta}(\xi\circ(s h_\gamma))(1-s)\,ds\,h_\gamma, h_\gamma\right\rangle\left(\int_0^1|ta+(1-t)b|^{p-2}\,dt\right)\,d\gamma\nonumber\\
& {=} (p-1)\int_{B_r^+(\xi)}\left\langle -\int_0^1\nabla^{2,*}_\hn u_{\varepsilon,\delta}(\xi\circ(s h_\gamma))(1-s)\,ds\,h_\gamma, h_\gamma\right\rangle\left(\int_0^1|ta+(1-t)b|^{p-2}\,dt\right)\,d\gamma,
\end{align}
\endgroup where $B_r^+(\xi):=\{\eta\in B_r(\xi):\nabla^2_\xi u_{\varepsilon,\delta}(\eta)\geq 0\}$.

Since \(u_{\varepsilon,\delta}\) is locally Lipschitz and
\(\nabla_{\mathbb H^N}u_{\varepsilon,\delta}\) is locally bounded, there exists
\(C>0\) such that
$$|a|=|u_{\varepsilon,\delta}(\xi)-u_{\varepsilon,\delta}(\gamma)|\le C|\gamma^{-1}\circ\xi|_{\mathbb H^N},$$ and
$$
|b|
=
\bigl|
\langle\nabla_{\mathbb H^N}u_{\varepsilon,\delta}(\xi),z_\gamma\rangle
\bigr|
\le
C|\gamma^{-1}\circ\xi|_{\mathbb H^N}.
$$
Consequently,
$$ |ta-(1-t)b| \le t|a|+(1-t)|b|\le C|\gamma^{-1}\circ\xi|_{\mathbb H^N},\qquad 0\le t\le 1.$$
\begin{equation}\label{eq:i}
0\le \int_0^1|ta-(1-t)b|^{p-2}\,dt \le C|\gamma^{-1}\circ\xi|_{\mathbb H^N}^{p-2}.
\end{equation}

Using (\ref{eq:i}), we get from (\ref{eq:h}) that,
\begin{equation}\label{eq:j}
    \mathcal{I}_{B_{{r}}(\xi)} \geq -C \int_{B_{{r}}(\xi)} \frac{|\xi^{-1} \circ \gamma|_{\mathbb{H}^N}^p}{|\xi^{-1} \circ \gamma|_{\mathbb{H}^N}^{Q+sp}} \, d\gamma \geq -C,
\end{equation}
for $p \geq 2$, where $C$ is independent of $\delta$. 

If \(\frac{2}{2-s}<p<2\), then Theorem~\ref{thm:taylor} yields
\[
|L(a)-L(b)|
\le
C|a-b|^{p-2},
\]
where \(L(t)=|t|^{p-2}t\). Therefore, by \eqref{eq:h},
\[
\mathcal J_{B_r(\xi)}
\ge
-
C
\int_{B_r(\xi)}
\frac{|\gamma^{-1}\circ\xi|_{\mathbb H^N}^{2(p-1)}}
{|\gamma^{-1}\circ\xi|_{\mathbb H^N}^{Q+sp}}
\,d\gamma .
\]
Using polar coordinates on \(\mathbb H^N\), we deduce
\[
\mathcal J_{B_r(\xi)} \ge
-
C
\int_0^r \rho^{2(p-1)-sp-1}\,d\rho \ge
-C,
\]
since
$2(p-1)-sp>0\quad\Longleftrightarrow\quad p>\frac{2}{2-s}.$
Hence, we have,
\begin{gather}\label{eq:k}
    \mathcal{L}u_{\varepsilon,\delta}(\xi) \geq -C, \, \xi \in \mathcal{K},
\end{gather}
with $C>0$ independent of $\delta$. Finally, applying Fatou's Lemma, we have the proof of Claim II.

\textbf{\underline{Claim III}}. We will prove the following:
\begin{equation*}
    \int_{\mathcal{K}} \liminf_{\delta \to 0} \mathcal{L}u_{\varepsilon, \delta}(\xi) \Psi(\xi)\, d\xi \geq \int_{\mathcal{K}} \mathcal{L}u_\varepsilon(\xi) \Psi(\xi)\, d\xi.
\end{equation*}
\textit{Proof of Claim III}. As in Claim II, let $0<{r}< dist(\mathcal{K}, \partial \Omega_{r(\varepsilon)})$. Then, $B_{{r}}(\xi) \subset \Omega_{r(\varepsilon)}$ and for $\gamma \in \mathbb{H}^N \setminus B_{{r}}(\xi)$, we have, $|\xi \circ \gamma^{-1}|_{\mathbb{H}^N}> {r}$. We note that, as in Claim II,
\begin{align*}
\mathcal{L}u_{\varepsilon, \delta}(\xi)&=\int_{\mathbb{H}^N \setminus B_r(\xi)} {|u_{\varepsilon,\delta}(\xi)-u_{\varepsilon,\delta}(\gamma)|^{p-2}
(u_{\varepsilon,\delta}(\xi)-u_{\varepsilon,\delta}(\gamma))} K(\xi,\gamma)\,d\gamma\nonumber\\
&\qquad+\int_{B_r(\xi)} {|u_{\varepsilon,\delta}(\xi)-u_{\varepsilon,\delta}(\gamma)|^{p-2}
(u_{\varepsilon,\delta}(\xi)-u_{\varepsilon,\delta}(\gamma)) } K(\xi,\gamma)\, d\gamma\nonumber\\
&=\int_{\mathbb{H}^N \setminus B_{{r}}(\xi)}
{|u_{\varepsilon,\delta}(\xi)-u_{\varepsilon,\delta}(\gamma)|^{p-2}
(u_{\varepsilon,\delta}(\xi)-u_{\varepsilon,\delta}(\gamma)) }K(\xi,\gamma)\, d\gamma\nonumber\\
&\qquad+\int_{B_{{r}}(\xi)} L\Big( u_{\varepsilon,\delta}(\xi)-u_{\varepsilon,\delta}(\gamma) \Big)-{L} \Big( -( \nabla_{\mathbb{H}^N} u_{\varepsilon, \delta}(\xi), \partial_t u_{\varepsilon, \delta}(\xi) \Big)
\cdot (\xi^{-1}\circ\gamma)\nonumber\\
&\qquad \times K(\xi,\gamma)\,d\gamma.
\end{align*}
As $u_\varepsilon$ is bounded, so analogous to Claim I, we have, 
\begin{align}\label{eq:l}
 &\int_{\mathbb{H}^N \setminus B_{{r}(\xi)}} {|u_{\varepsilon,\delta}(\xi)-u_{\varepsilon,\delta}(\gamma)|^{p-2}
(u_{\varepsilon,\delta}(\xi)-u_{\varepsilon,\delta}(\gamma)) }K(\xi,\gamma)\, d\gamma\nonumber\\
&\leq C \int_{\mathbb{H}^N \setminus B_r(\xi)} \frac{1+|u_\varepsilon(\gamma)|^{p-1}}{|\xi \circ \gamma^{-1}|_{\mathbb{H}^N}^{Q+sp}} \, d\gamma 
\end{align}
Again, similar to (\ref{eq:j})-(\ref{eq:k}), we get,
\begin{align*}
    &\int_{B_{{r}}(\xi)} {L}\Big( u_{\varepsilon,\delta}(\xi)-u_{\varepsilon,\delta}(\gamma) \Big)-{L} \Big( -( \nabla_{\mathbb{H}^N} u_{\varepsilon, \delta}(\xi), \partial_t u_{\varepsilon, \delta}(\xi) \Big)
\cdot (\xi^{-1}\circ\gamma) K(\xi,\gamma)\,d\gamma\\
&\geq -C \begin{cases}
\int_{B_{{r}}(\xi)} |\xi \circ \gamma^{-1}|_{\mathbb{H}^N}^{p(1-s)-Q}\, d\gamma, \, p \geq 2,\\
\int_{B_{{r}}(\xi)} |\xi \circ \gamma^{-1}|_{\mathbb{H}^N}^{2(p-1)-Q-sp}\, d\gamma, \, \frac{2}{2-s}<p<2.\\
\end{cases}\\
&\geq -C.
\end{align*}
Using this, we have, from Fatou's Lemma,
\begin{gather*}
    \liminf_{\delta \to 0} \mathcal{L}u_{\varepsilon, \delta}(\xi) \geq \mathcal{L}u_{\varepsilon}(\xi).
\end{gather*}
This gives us Claim III. Combining Claim I-III, we have
\begin{equation}\label{eq:9}
\begin{aligned}
\liminf_{\delta\to0}
\int_{K_0}\mathcal L u_{\varepsilon,\delta}(\xi)\Psi(\xi)\,d\xi
&\ge
\int_{K_0}
\liminf_{\delta\to0}\mathcal L u_{\varepsilon,\delta}(\xi)\Psi(\xi)\,d\xi \\
&\ge
\int_{K_0}\mathcal L u_\varepsilon(\xi)\Psi(\xi)\,d\xi .
\end{aligned}
\end{equation}
Since
\(\mathcal L u_\varepsilon(\xi)\ge f_\varepsilon(\xi)\,\text{for a.e. }\xi\in K_0,\)
we obtain
\begin{equation}\label{eq:10}
    \liminf_{\delta\to0}
\int_{K_0}\mathcal L u_{\varepsilon,\delta}(\xi)\Psi(\xi)\,d\xi
\ge
\int_{K_0} f_\varepsilon(\xi)\Psi(\xi)\,d\xi .
\end{equation}
Finally, taking the limit inferior in \eqref{eq:a}, and using \eqref{eq:b} and \eqref{eq:10},
we conclude
\[
\iint_{Q_{K_0}}
|u_\varepsilon(\xi)-u_\varepsilon(\gamma)|^{p-2}
\bigl(u_\varepsilon(\xi)-u_\varepsilon(\gamma)\bigr)
\bigl(\Psi(\xi)-\Psi(\gamma)\bigr)
K(\xi,\gamma)\,d\gamma d\xi
\ge
\int_{K_0}f_\varepsilon(\xi)\Psi(\xi)\,d\xi .
\]
Since \(K_0=\operatorname{supp}\Psi\), this proves the lemma.
\end{proof}

We are now in a position to derive the weak formulation satisfied by the regularized viscosity supersolutions.
\begin{lemma}\label{lem:5} (Uniform Caccioppoli-type estimate)
Let \(u\in L^\infty(\mathbb H^N)\) be a weak supersolution of
\(\mathcal Lu=f\) in \(\Omega\), where \(f\in C(\Omega)\). 
Let \(\phi\in C_c^\infty(\Omega)\), \(0\leq \phi\leq 1\), and put
\(K:=\operatorname{supp}\phi\). Then there exists a constant
\(C>0\), depending only on \(p,s,Q,\lambda,\Lambda,\phi\) and
\(\|f\|_{L^\infty(K)}\), such that
\[
\int_K\int_{\mathbb H^N}
\frac{|u(\xi)-u(\eta)|^p}{|\eta^{-1}\circ\xi|_{\mathbb H^N}^{Q+sp}}
\phi^p(\xi)\,d\eta\,d\xi
\leq
C\left[
\operatorname{osc}(u)^p
\iint_{Q_K}
\frac{|\phi(\xi)-\phi(\eta)|^p}
{|\eta^{-1}\circ\xi|_{\mathbb H^N}^{Q+sp}}
\,d\eta\,d\xi
+\operatorname{osc}(u)
\right],
\]
where
\(\operatorname{osc}(u):=\sup_{\mathbb H^N}u-\inf_{\mathbb H^N}u .\)
\end{lemma}
\begin{proof}
Let \(\phi\in C_c^\infty(\Omega)\), \(0\leq \phi\leq 1\), and put
\(K:=\operatorname{supp}\phi\). Set
\[
M:=\sup_{\mathbb H^N}u,\, m:=\inf_{\mathbb H^N}u,\, \operatorname{osc}(u):=M-m .
\]
Define
\(
\varphi(\xi):=(M-u(\xi))\phi^p(\xi).
\)
Then \(\varphi\geq0\), \(\operatorname{supp}\varphi\subset K\), and
\(\varphi\in W^{s,p}_0(\Omega)\). Hence \(\varphi\) is an admissible test function in the weak supersolution inequality. Therefore,
\[
\int_K f(\xi)(M-u(\xi))\phi^p(\xi)\,d\xi
\leq
\iint_{Q_K}
|u(\xi)-u(\eta)|^{p-2}(u(\xi)-u(\eta))
(\varphi(\xi)-\varphi(\eta))
K(\xi,\eta)\,d\eta\,d\xi .
\]
Using the identity
\[
\varphi(\xi)-\varphi(\eta)
=
-(u(\xi)-u(\eta))\phi^p(\xi)
+
(\phi^p(\xi)-\phi^p(\eta))(M-u(\eta)),
\]
we get
\[
\begin{aligned}
&\int_K\int_{\mathbb H^N}
|u(\xi)-u(\eta)|^p\phi^p(\xi)K(\xi,\eta)\,d\eta\,d\xi\\
&\leq
\iint_{Q_K}
|u(\xi)-u(\eta)|^{p-1}
|\phi^p(\xi)-\phi^p(\eta)| (M-u(\eta))K(\xi,\eta)\,d\eta\,d\xi
-\int_K f(\xi)(M-u(\xi))\phi^p(\xi)\,d\xi .
\end{aligned}
\]
Since
\(0\leq M-u(\eta)\leq \operatorname{osc}(u),\)
and
\[
|\phi^p(\xi)-\phi^p(\eta)|
\leq
C_p|\phi(\xi)-\phi(\eta)|
\bigl(\phi^{p-1}(\xi)+\phi^{p-1}(\eta)\bigr),
\]
Young's inequality gives, for every \(\delta>0\),
\[
\begin{aligned}
\int_K\int_{\mathbb H^N}
|u(\xi)-u(\eta)|^p\phi^p(\xi)K(\xi,\eta)\,d\eta\,d\xi
&\leq
\delta C_p
\iint_{Q_K}
|u(\xi)-u(\eta)|^p
\bigl(\phi^p(\xi)+\phi^p(\eta)\bigr)
K(\xi,\eta)\,d\eta\,d\xi \\
&\quad
+C_p\delta^{1-p}(\operatorname{osc}(u))^p
\iint_{Q_K}
|\phi(\xi)-\phi(\eta)|^pK(\xi,\eta)\,d\eta\,d\xi \\
&\quad
-\int_K f(\xi)(M-u(\xi))\phi^p(\xi)\,d\xi .
\end{aligned}
\]
Since \(\phi\equiv0\) outside \(K\), we have
\[
\iint_{Q_K}|u(\xi)-u(\eta)|^p\phi^p(\xi)K(\xi,\eta)\,d\eta d\xi
=
\int_K\int_{\mathbb H^N}|u(\xi)-u(\eta)|^p\phi^p(\xi)K(\xi,\eta)\,d\eta d\xi.
\]
Moreover, by the symmetry of \(K\),
\[
\iint_{Q_K}|u(\xi)-u(\eta)|^p\phi^p(\eta)K(\xi,\eta)\,d\eta d\xi
=
\int_K\int_{\mathbb H^N}|u(\xi)-u(\eta)|^p\phi^p(\xi)K(\xi,\eta)\,d\eta d\xi.
\]

Choosing \(\delta>0\) sufficiently small, the first term on the right-hand side can be absorbed into the left-hand side. Thus,
\[
\begin{aligned}
\int_K\int_{\mathbb H^N}
|u(\xi)-u(\eta)|^p\phi^p(\xi)K(\xi,\eta)\,d\eta\,d\xi
&\leq
C(\operatorname{osc}(u))^p
\iint_{Q_K}
|\phi(\xi)-\phi(\eta)|^pK(\xi,\eta)\,d\eta\,d\xi \\
&\quad
-\int_K f(\xi)(M-u(\xi))\phi^p(\xi)\,d\xi .
\end{aligned}
\]
Since \(f\in C(\Omega)\) and \(K\Subset\Omega\), we have \(f\in L^\infty(K)\). Therefore,
\[
\left|
\int_K f(\xi)(M-u(\xi))\phi^p(\xi)\,d\xi
\right|
\leq
\|f\|_{L^\infty(K)}\,|K|\,\operatorname{osc}(u).
\]
Hence,
\[
\begin{aligned}
\int_K\int_{\mathbb H^N}
|u(\xi)-u(\eta)|^p\phi^p(\xi)K(\xi,\eta)\,d\eta\,d\xi
&\leq
C(\operatorname{osc}(u))^p
\iint_{Q_K}
|\phi(\xi)-\phi(\eta)|^pK(\xi,\eta)\,d\eta\,d\xi \\
&\quad
+C\,\operatorname{osc}(u).
\end{aligned}
\]
Finally, applying the lower and upper ellipticity bounds on \(K\), we have
\[
\frac{\lambda}{|\eta^{-1}\circ\xi|_{\mathbb H^N}^{Q+sp}}
\leq
K(\xi,\eta)
\leq
\frac{\Lambda}{|\eta^{-1}\circ\xi|_{\mathbb H^N}^{Q+sp}},
\]
and hence, we obtain
\[
\begin{aligned}
\int_K\int_{\mathbb H^N}
\frac{|u(\xi)-u(\eta)|^p}{|\eta^{-1}\circ\xi|_{\mathbb H^N}^{Q+sp}}
\phi^p(\xi)\,d\eta\,d\xi
&\leq
C(\operatorname{osc}(u))^p
\iint_{Q_K}
\frac{|\phi(\xi)-\phi(\eta)|^p}
{|\eta^{-1}\circ\xi|_{\mathbb H^N}^{Q+sp}}
\,d\eta\,d\xi  \\
&\quad
+C\,\operatorname{osc}(u),
\end{aligned}
\]
where \(C>0\) depends only on \(p,s,Q,\lambda,\Lambda,\phi\) and
\(\|f\|_{L^\infty(K)}\), but is independent of \(u\). This proves the lemma.
\end{proof}

We now combine the approximation, weak formulation, and local energy estimates established in the preceding lemmas to prove the equivalence between weak and viscosity solutions.
\subsection{Proof of Theorem \ref{maintheorem2}}: 
We prove the assertion for viscosity supersolutions. The proof for subsolutions is analogous.

Let \(u\in L^\infty(\mathbb H^N)\) be a viscosity supersolution of
\[
\mathcal Lu=f \qquad \text{in }\Omega .
\]
For \(\varepsilon>0\), let \(u_\varepsilon\) be the modified infimal convolution
\[
u_\varepsilon(\xi)
:=
\inf_{\eta\in\mathbb H^N}
\left\{
u(\eta)
+
\frac{|\eta\circ\xi^{-1}|_{\mathbb H^N}^{q}}
{q\varepsilon^{q-1}}
\right\}.
\]
By Lemma \ref{lem:inf-conv}, \(u_\varepsilon\) is a viscosity supersolution of
\[
\mathcal L u_\varepsilon=f_\varepsilon
\qquad\text{in }\Omega_{r(\varepsilon)},
\]
where
\[
f_\varepsilon(\xi)
:=
\inf_{\eta\in B(0,r(\varepsilon))} f(\eta\circ\xi).
\]
Moreover, since \(p>\frac{2}{2-s}\), the principal value is well-defined for the
test functions used in Lemma \ref{lem:inf-conv}, and therefore
\[
\mathcal L u_\varepsilon(\xi)\ge f_\varepsilon(\xi)
\qquad\text{for a.e. }\xi\in\Omega_{r(\varepsilon)}.
\]
Hence, by Lemma \ref{lem:weak-form-ueps}, \(u_\varepsilon\) is a weak supersolution in
\(\Omega_{r(\varepsilon)}\). Thus, for every non-negative
\(\Psi\in C^\infty_0(\Omega_{r(\varepsilon)})\), we have
\begin{equation}\label{eq:A}
\iint_{Q_{\operatorname{supp}\Psi}}
|u_\varepsilon(\xi)-u_\varepsilon(\zeta)|^{p-2}
\bigl(u_\varepsilon(\xi)-u_\varepsilon(\zeta)\bigr)
\bigl(\Psi(\xi)-\Psi(\zeta)\bigr)
K(\xi,\zeta)\,d\zeta d\xi
\ge
\int_{\operatorname{supp}\Psi}
f_\varepsilon(\xi)\Psi(\xi)\,d\xi .
\end{equation}

Let now \(\Phi\in C^\infty_0(\Omega)\), \(\Phi\ge0\). Choose \(\varepsilon>0\) sufficiently
small so that
\[
\operatorname{supp}\Phi\Subset \Omega_{r(\varepsilon)}.
\]
Taking \(\Psi=\Phi\) in \eqref{eq:A}, we obtain
\begin{equation}\label{eq:B}
\iint_{Q_{\operatorname{supp}\Phi}}
|u_\varepsilon(\xi)-u_\varepsilon(\zeta)|^{p-2}
\bigl(u_\varepsilon(\xi)-u_\varepsilon(\zeta)\bigr)
\bigl(\Phi(\xi)-\Phi(\zeta)\bigr)
K(\xi,\zeta)\,d\zeta d\xi
\ge
\int_{\operatorname{supp}\Phi}
f_\varepsilon(\xi)\Phi(\xi)\,d\xi .
\end{equation}
Since \(f\in C(\Omega)\), and \(r(\varepsilon)\to0\), we have
\(f_\varepsilon(\xi)\to f(\xi) \, \text{uniformly on }\operatorname{supp}\Phi.\)
Therefore,
\begin{equation}
    \int_{\operatorname{supp}\Phi}f_\varepsilon(\xi)\Phi(\xi)\,d\xi \to
\int_{\operatorname{supp}\Phi}f(\xi)\Phi(\xi)\,d\xi .
\end{equation}
It remains to pass to the limit in the left-hand side of \eqref{eq:B}. We first record a
uniform energy estimate. Choose \(\widetilde\Phi\in C^\infty_0(\Omega)\), \(0\le
\widetilde\Phi\le1\), such that
\[
\widetilde\Phi\equiv1
\quad\text{on a compact set }K
\quad\text{with}\quad
\operatorname{supp}\Phi\Subset K\Subset\Omega .
\]
For \(\varepsilon>0\) sufficiently small, \(K\Subset\Omega_{r(\varepsilon)}\). Applying
Lemma \ref{lem:5} to \(u_\varepsilon\) with the test function \(\widetilde\Phi\), we obtain
\begin{equation}\label{eq:C}
    \iint_{Q_K}
|u_\varepsilon(\xi)-u_\varepsilon(\zeta)|^p
K(\xi,\zeta)\,d\zeta d\xi
\le C,
\end{equation}
where \(C>0\) is independent of \(\varepsilon\).

Since \(u_\varepsilon\uparrow u\) pointwise a.e. in \(\mathbb H^N\) as
\(\varepsilon\to0\), we have
\[
u_\varepsilon(\xi)-u_\varepsilon(\zeta)
\to
u(\xi)-u(\zeta)
\qquad\text{for a.e. }(\xi,\zeta)\in\mathbb H^N\times\mathbb H^N.
\]
Moreover, by \eqref{eq:C}, the family
\[
|u_\varepsilon(\xi)-u_\varepsilon(\zeta)|^{p-1}
|\Phi(\xi)-\Phi(\zeta)|K(\xi,\zeta)
\]
is uniformly integrable on \(Q_{\operatorname{supp}\Phi}\). Indeed, by H\"older's
inequality,
\begin{equation}\label{eq:D}
\begin{aligned}
&\iint_{Q_{\operatorname{supp}\Phi}}
|u_\varepsilon(\xi)-u_\varepsilon(\zeta)|^{p-1}
|\Phi(\xi)-\Phi(\zeta)|K(\xi,\zeta)\,d\zeta d\xi  \\
&\qquad\le
\left(
\iint_{Q_K}
|u_\varepsilon(\xi)-u_\varepsilon(\zeta)|^p
K(\xi,\zeta)\,d\zeta d\xi
\right)^{\frac{p-1}{p}}
\left(
\iint_{Q_{\operatorname{supp}\Phi}}
|\Phi(\xi)-\Phi(\zeta)|^p
K(\xi,\zeta)\,d\zeta d\xi
\right)^{\frac 1p}
\le C.
\end{aligned}
\end{equation}
Since \(\Phi\in C^\infty_0(\Omega)\), the second integral above is finite. Hence, by
Vitali's convergence theorem,
\begin{equation}\label{eq:E}
\begin{aligned}
&\iint_{Q_{\operatorname{supp}\Phi}}
|u_\varepsilon(\xi)-u_\varepsilon(\zeta)|^{p-2}
\bigl(u_\varepsilon(\xi)-u_\varepsilon(\zeta)\bigr)
\bigl(\Phi(\xi)-\Phi(\zeta)\bigr)
K(\xi,\zeta)\,d\zeta d\xi \\
&\qquad\to
\iint_{Q_{\operatorname{supp}\Phi}}
|u(\xi)-u(\zeta)|^{p-2}
\bigl(u(\xi)-u(\zeta)\bigr)
\bigl(\Phi(\xi)-\Phi(\zeta)\bigr)
K(\xi,\zeta)\,d\zeta d\xi .
\end{aligned}
\end{equation}

Passing to the limit in \eqref{eq:B}, using \eqref{eq:C} and \eqref{eq:E}, we obtain
\[
\iint_{Q_{\operatorname{supp}\Phi}}
|u(\xi)-u(\zeta)|^{p-2}
\bigl(u(\xi)-u(\zeta)\bigr)
\bigl(\Phi(\xi)-\Phi(\zeta)\bigr)
K(\xi,\zeta)\,d\zeta d\xi
\ge
\int_{\operatorname{supp}\Phi}f(\xi)\Phi(\xi)\,d\xi .
\]
Since \(\Phi\in C^\infty_0(\Omega)\), \(\Phi\ge0\), was arbitrary, this proves that
\(u\) is a weak supersolution of
\[
\mathcal Lu=f
\qquad\text{in }\Omega .
\]
The proof for viscosity subsolutions is obtained by applying the previous argument to \(-u\). Therefore, viscosity solutions are weak solutions. This completes the proof.

\end{document}